\documentclass[letterpaper, 10pt, conference]{ieeeconf}
\usepackage{amsmath,amssymb,amsfonts}
\usepackage{graphicx}
\usepackage{textcomp}
\usepackage{graphicx}      % include this line if your document contains figures
\usepackage{algorithm}
\usepackage{algorithmicx}
\usepackage{bm}            % bold greek letters
\usepackage{amsfonts}
\usepackage{mathtools}      
\usepackage{mathrsfs}

\usepackage{url}
\usepackage{hyperref}  
\usepackage{lettrine}
\usepackage{comment}

\newcommand{\p}{\bm{\mu}}      % Scheduling signal process
\newcommand{\pdim}{n_{\mu}}         % Dimension of the scheduling signal
\newcommand{\nx}{n_x}      
\newcommand{\ny}{n_y}

\newcommand{\udim}{n_u}
\newcommand{\rank}{\mathrm{rank}}

\newcommand{\msig}{\p_{\sigma}}
\newcommand{\psig}{p_{\sigma}}
\newcommand{\Psig}{P_{\sigma}}
\newcommand{\Qsig}{Q_{\sigma}}

\newcommand{\sigSet}{\sigma \in \Sigma}
\newcommand{\word}{\sigma_{1}\sigma_{2}\cdots \sigma_{k}}
\newcommand{\wordSet}[1]{w \in \Sigma^{#1}}
\newcommand{\expect}[1]{E\left[#1 \right]}
\newcommand{\filt}{\mathcal{F}_{t}^{\p,-}}
\newcommand{\filtp}{\mathcal{F}_{t}^{\p,+}}

\newcommand{\filtr}{\mathcal{F}_{t}^{\mathbf{r}}}

\newcommand{\uw}[1]{\p_{#1}}
\newcommand{\zwy}{\mathbf{z}^{\mathbf{y}}_{w} }

\newcommand{\zwr}{\mathbf{z}^{\mathbf{r}}_{w} }
\newcommand{\zvr}{\mathbf{z}^{\mathbf{r}}_{v} }
\newcommand{\zsr}{\mathbf{z}^{\mathbf{r}}_{s} }

\newcommand{\zsigr}{\mathbf{z}^{\mathbf{r}}_{\sigma}}

\newcommand{\timeset}{t \in \mathbb{Z}}
\newcommand{\yb}{\mathbf{y}}

\newcommand{\vb}{\mathbf{v}}
\newcommand{\eb}{\mathbf{e}}
\newcommand{\xb}{\mathbf{x}}
\newcommand{\ub}{\mathbf{u}}
\newcommand{\z}{\mathbf{z}}
\renewcommand{\r}{\mathbf{r}}

\newcommand{\Bs}{G}   % Stochastic B matrix

\newcommand{\covseq}{\Psi_{\mathbf{y}}}

\newtheorem{Notation}{Notation}
\newtheorem{Definition}{Definition}
\newtheorem{Example}{Example}
\newtheorem{Theorem}{Theorem}
\newtheorem{Corollary}{Corollary}
\newtheorem{Lemma}{Lemma}
\newtheorem{Remark}{Remark}

\let\labelindent\relax
\usepackage{enumitem}
\renewcommand{\theenumi}{\arabic{enumi}}

\renewcommand{\theenumii}{\arabic{enumii}}

\renewcommand{\theenumiii}{\arabic{enumiii}}

\renewcommand{\theenumiv}{\arabic{enumiv}}

\newcommand{\btheta}{\boldsymbol{\theta}}

\newcommand{\GBS}{\textbf{GBS}}

\newcommand{\y}{\textbf{y}}
\newcommand{\bu}{\textbf{u}}
\newcommand{\x}{\textbf{x}}
\newcommand{\bmu}{\bm{\mu}}

\newcommand{\e}{\textbf{e}}
\newcommand{\IM}{\bm{Im}}
\newcommand{\AQ}{\mathbb{I}_{0,n_p}}

\title{ On minimal LPV state-space representations in innovation form: an algebraic characterization}
\author{Elie Rouphael, Mihaly Petreczky,  Lotfi Belkoura
\thanks{E. Rouphael, M. Petreczky and L. Belkoura are with the Centre de Recherche en Informatique, Signal et Automatique de Lille, UMR CNRS 9189, Villeneuve dAscq 59651, France
{\tt\small : \{elie.rouphael,mihaly.petreczky, lotfi.belkoura\}@univ-lille.fr},
}}
\begin{document}
\maketitle
\begin{abstract}
 In this paper a definition of the concept of minimal state-space representations in innovation form for LPV is proposed. We also present algebraic conditions for a stochastic LPV state-space representation to be minimal in forward innovation form and discuss an algorithm for transforming any stochastic LPV state-space representation to a minimal one
 in innovation form.
\end{abstract}

\section{Introduction}
\label{sect:intro}
Identification  of \emph{Linear Parameter-Varying} (LPV) models has gained significant attention, see \cite{bagi,laurainrefined,MEJARIAuto,pigalpv,fewive,tanelliidentification,Wingerden09,veve,toth} and the references therein.
%detailed summary of the available LPV identification approaches.
In particular, there is a rich literature on subspace identification of LPV state-space representations, see for instance \cite{Wingerden09,fewive,Verdult02,veve,CoxTothSubspace,RamosSubspace} and the references therein.

Despite these advances, the theoretical analysis of 
system identification algorithms, especially subspace methods,  for stochastic LPV state-space representations remains challenging. 
%This is particularly true for subspace identification algorithms. 
As the history of LTI system identification indicates \cite{LindquistBook,Katayama:05,CHIUSO2005377}
that such a theoretical analysis requires a good understanding of the notion of minimal state-space representations in innovation form \cite{LindquistBook,Katayama:05}. 
The latter notion is not yet fully understood for LPV state-space representations. 
%For LPV systems, such a  nocharacterization has been missing. 
%which have the following useful properties:
%\begin{itemize}
%    \item{\textbf{P1}} For an output process, all its minimal LTI state-space representations in innovation form are  isomorphic.
%    \item{\textbf{P2}}  If a process has a LTI state-space representation, then it also has a minimal one in forward innovation
%          form.
%    \item{\textbf{P3}}  State-space representations in innovation form can be viewed as devices for
%          predicting future outputs based on past outputs and inputs, and the noise process of such systems
%          can be interpreted as the corresponding prediction error.
%\end{itemize}

\paragraph*{\textbf{Contribution}}
In this paper LPV state-space representations with affine dependence on parameters (abbreviated by \emph{LPV-SSA}) are considered.
We restrict attention to autonomous (without control input) stationary stochastic LPV-SSAs (\emph{asLPV-SSA} for short). 
\emph{The main technical contributions are new algebraic conditions for 
an asLPV-SSA to be minimal in innovation form}.
These conditions depend only the the matrices of the system representation. 
In order to present these results, the paper
also provides a \emph{systematic overview of the results
          on existence and minimality of asLPV-SSA in innovation form}.
          These results can be derived 
          from realization theory of stochastic bilinear
          systems \cite{PetreczkyBilinear}, but they have not been stated
          explicitly for LPV-SSAs.
          %and hence they are not easily accessible
          %to the LPV research community. 
          In particular, we state that any asLPV-SSA can be converted to a 
          minimal asLPV-SSA in innovation form while preserving the output. Moreover, any two minimal asLPV-SSA in innovation form are isomorphic, if they have the same output

\paragraph*{\textbf{Motivation for studying asLPV-SSA}}
Under suitable technical assumptions, any stochastic LPV-SSA can be decomposed into a noiseless deterministic LPV-SSA which is driven only by the control input, and asLPV-SSA driven by the noise, \cite{MejariLPVS2019}. 
Moreover, the identification of these two subsystems can be carried out separately \cite{MejariLPVS2019}.
Note that minimality and uniqueness of deterministic
LPV state-space representations is well understood \cite{PetreczkyLPVSS}.
Hence, in order to understand the notion of minimality and innovation representation for stochastic LPV-SSA with control input, the first step is to understand these notions for asLPV-SSA. 

\paragraph*{\textbf{Motivation for minimal LPV-SSAs in innovation form}}
%As it was mentioned above, the notion of minimal asLPV-SSA in innovation form is useful for analysis of subspace identification algorithms.
%Controller design approaches (e.g., \cite{Carsten,zhou})
%often require the LPV models to be in SS representation with an affine dependency on the scheduling variable. To this end, \emph{realization theory} of LPV models plays a key role in understanding the conditions under which the observed behavior of a system can be realized by a state-space affine representation. It also allows one to formulate identification algorithms for estimating state-space representation from a finite set of observations. 
%While realization theory for deterministic  (LPV-SSA) representation has been developed in~\cite{RolandAbbas,PetreczkyLPVSS,toth}, the case of stochastic LPV systems remains largely open. 
%In turn, this makes it difficult to formally analyze subspace identification algorithms. In fact, even the formulation of the subspace identification problem brings up mathematical challenges. 
In the formulation of the identification problem for LPV-SSA
%representations with affine static dependence  (LPV-SSA for short)
\cite{Wingerden09,fewive,Verdult02,veve,CoxTothSubspace}
the stated goal is usually to find an LPV-SSA which is isomorphic\footnote{Often, it is required that the isomorphism does not depend on the scheduling.}, 
or at least
which is input-output equivalent to the data generating system. 
However, in general there may exist LPV-SSAs which generate the same
output for some scheduling signal, but which are not
isomorphic (see Example \ref{sim:examp2} in Section \ref{sect:examp} or \cite{TothKulcsar}), or which are not input-output equivalent.
In the latter case the systems generate different outputs when the scheduling
signal is changed. See Example \ref{sim:examp} and Example \ref{sim:examp2}
of Section \ref{sect:examp} of the present paper.

%If instead of requiring that the model and the data generating system are isomorphic we require
%only that they have the same output for same scheduling signals and input, the
%problem remains ill-posed. Indeed, there exist LPV-SSAs
%which have the same output for a certain choice of the scheduling signal, and different
%output for other choices of the scheduling signal, even for white noise (hence rich) scheduling signals. In Example \ref{sim:examp} and Example \ref{sim:examp2}
%of Section \ref{sect:examp} we will present such examples.  
%Note that in Example \ref{sim:examp}-\ref{sim:examp2}
%all the scheduling signals are white noises, so the phenomenon is not due to the signal not
%being rich enough. 

This implies that in general the system identification problem
is ill-posed.  However, it becomes well-posed, if we add the assumption that the underlying data generating system is a minimal asLPV-SSA in innovation form. 

Indeed, if the observed output has an asLPV-SSA representation, it has a minimal one in innovation form, and all such representations are related by a constant
isomorphism.
The necessity of this assumption is illustrated by the examples in Section \ref{sect:examp}.
From \cite{MejariCDC2019} it follows
that this assumption is sufficient, as the identification algorithm in \cite{MejariCDC2019} returns a minimal asLPV-SSA in innovation form. 
We conjecture that the same will be true for
most of the existing subspace identification algorithms  \cite{Wingerden09,fewive,Verdult02,veve,CoxTothSubspace,RamosSubspace}.

%This problem will be illustrated by Example \ref{sim:examp} in Section \ref{sect:examp}. There we present
%two stochastic LPV-SSA systems, such that they are not isomorphic, but generate the same
%output for a certain choice of noise and the scheduling signal. Moreover, if the scheduling
%µsignal is changed, then the two systems cease to generate the same output, even if the
%noise processes remain unchanged. The issue of Example \ref{sim:examp} is that one of the systems
%s not minimal, and hence not isomorphic to the second one.
%Another example, namely Example \ref{sim:examp2} illustrates that minimality is not sufficient.
%There both systems are minimal, yet they are not isomorphic. In fact, those two systems
%generate the same output for a certain choice of scheduling signal. However, if the scheduling
%signal is changed, then the corresponding outputs will not be the same.
%Finally, in Example \ref{} we present examples of two asLPV-SSAs which satisfy the algebraic
%conditions formulated in this paper and hence are both minimal and in forward innovation
%form. These two systems 

%e 
To deal only with minimal 
asLPV-SSA representations in innovation form, simple conditions to check minimality and being in innovation form are needed.
The latter is necessary in order to check if the elements of a
parametrization of asLPV-SSAs are minimal and in innovation form, or to construct such parametrizations.

\paragraph*{\textbf{Related work}}
As it was mentioned above, there is a rich literature on 
subspace identification methods for stochastic LPV-SSA representations \cite{RamosSubspace,FavoreelTAC,CoxTothSubspace,Wingerden09}. However, the cited papers do not deal with the problem
of characterizing minimal stochastic 
LPV state-space representations in innovation form. 
In \cite{CoxLPVSS,CoxTothSubspace} the existence of an LPV state-space representation in innovation form was studied, but
due to the specific assumptions (deterministic scheduling) and the
definition of the innovation process, the resulting LPV state-space
representation in innovation form had dynamic dependence on the
scheduling parameters. Moreover, \cite{CoxLPVSS,CoxTothSubspace}
do not address the issue of minimality of the stochastic part of LPV state-space representations. 

This paper uses realization theory of stochastic generalized bilinear systems (\emph{\GBS\ } for short) of \cite{PetreczkyBilinear}. In particular,
asLPV-SSAs correspond to \GBS{s}. The existence and uniqueness of minimal asLPV-SSA{s} in innovation form
follows from the results of \cite{PetreczkyBilinear}. 
%However,
%the algebraic characterization of minimal asLPV-SSA{s} in innovation form is new.
The main novelty of the present paper with respect to \cite{PetreczkyBilinear} is 
the new algebraic characterization of minimal asLPV-SSA{s} in innovation form, and 
that the results on   existence and uniqueness of minimal \GBS{s} are spelled out explicitly for LPV-SSAs.

The paper \cite{MejariCDC2019} used the correspondence
between \GBS{s} and asLPV-SSA{s} to state existence and uniqueness
of minimal asLPV-SSA{s} in innovation form. However, \cite{MejariCDC2019} did not provide an algebraic characterization of minimality or innovation form.  
%Instead, it focused on formulating a statistically consistent algorithm for
%estimating asLPV-SSA{s}.
Moreover, it considered only scheduling signals which were zero mean white noises. In contrast, in this paper more general scheduling signals are considered. 
The present paper is complementary to \cite{MejariCDC2019}. This paper
explains when the assumption that the data generating system is minimal asLPV-SSA  in innovation form could be true, while \cite{MejariCDC2019} presents an identification algorithm which is statistically consistent under the latter assumption.

\textbf{Outline of the paper}
In Section \ref{sect:prelim} we introduce the notations used and we recall \cite{PetreczkyBilinear}, some technical assumptions  which are necessary to define of the stationary LPV-SSA representation.
In Section \ref{sect:min} some principal results on minimal asLPV-SSA{s} in innovation form are reviewed. %, a minimizing algorithm to tranform an asLPV-SSA to a minimal one in innovation form is presented.
In Section \ref{sect:main} we present the main results of the paper, namely,
algebraic conditions for an asLPV-SSA to be minimal in innovation form.
Finally, in Section \ref{sect:examp} numerical examples are developed
to illustrate the contributions.

\section{Preliminaries}
\label{sect:prelim}
%We will present some technical notations and definitions regarding the involved processes, in order to help prepare the exposition of the main results.
In the sequel, we will use the standard terminology of probability theory \cite{Bilingsley}. In particular, all the random variables and stochastic processes are understood w.r.t. to a fixed probability space $\left(\Omega, \mathcal{F}, \mathcal{P}\right)$, where $\mathcal{F}$ is a $\sigma$-algebra over the sample space $\Omega$.
	%(\emph{i.e.}, $\mathcal{F}$ is a collection of subsets of $\Omega$, that includes $\Omega$ itself, is closed under complement, is closed under countable unions and is closed under countable intersections) and $\mathcal{P}$ is a probability measure on $\mathcal{F}$. For two $\sigma$-algebras $\mathcal{F}_i$, $i=1,2$, $\mathcal{F}_1 \lor \mathcal{F}_2$ denotes the smallest $\sigma$-algebra generated by the $\sigma$-algebras $\mathcal{F}_1,\mathcal{F}_2$. 
The expected value of a random variable $\mathbf{r}$ is denoted by $E[\mathbf{x}]$ and conditional expectation w.r.t. $\sigma$- algebra $\mathcal{F}$ is denoted by $\expect{\mathbf{r} \mid \mathcal{F}}$. 
All the stochastic processes in this paper are discrete-time ones defined over the time-axis $\mathbb{Z}$ of the set of integers:
a stochastic process $\r$ is a collection of random variables
$\{\r(t)\}_{t \in \mathbb{Z}}$ taking values in some set $X$.

Next, we define the class of systems studied in this paper. An \emph{autonomous stochastic 
linear parameter-varying state-space representation with affine 
dependence on scheduling parameter (aLPV-SSA)} a is system described by
 \begin{equation}
	\label{eq:aslpv}
            \mathcal{S} \left\{     \begin{aligned}
                 &  \xb(t+1)=\sum_{i=1}^{\pdim} (A_i\xb(t)+K_i\vb(t))\p_i(t) \\
                 & \y(t)=C\xb(t)+F\vb(t)
		\end{aligned}\right.
               \end{equation}
           where 
	$A_{\sigma} \in \mathbb{R}^{n \times n}, K_{\sigma} \in \mathbb{R}^{n \times m}$,
	$C \in \mathbb{R}^{p \times n}$, $F \in \mathbb{R}^{n_y \times m}$,
	and $\xb$ is the state process, $\p=\begin{bmatrix} \p_1,\ldots,\p_{\pdim} \end{bmatrix}^T$ is the 
	scheduling process,
	$\vb$ is the noise process and $\y$ is the output process.
	
	Note that all the involved processes are defined for both negative and positive time. This may create technical problems for
	the existence of a solution and the role of initial state. 
	In this paper we will circumvent this problem by considering
	LPV-SSA which are mean-square stable in a suitable sense and
	state process of which is stationary. In order to define this
	class of LPV-SSAs we will have to recall from \cite{PetreczkyBilinear} some notation and terminology.

%A discrete-time stochastic process is a collection $\{\mathbf{r}(t)\}_{t \in \mathbb{Z}}$ taking values in $X$, where $\mathbf{r}(t) \in X$ is a random variable for all $t \in \mathbb{Z}$. Let $\mathbf{x}$ be a stochastic process taking values in $\mathbb{R}^{n}$; $\mathbf{x}$ is called \emph{zero mean}, if $\expect{\mathbf{x}(t)} = 0$ for all $t \in \mathbb{Z}$; $\mathbf{x}$  is called \emph{square integrable}, if the expectations $\expect{\mathbf{x}(t)}$, $\expect{\xb(t) \xb^{T}(t)}$ exist for all $t \in \mathbb{Z}$. A square-integrable process $\xb$ is called \emph{wide sense stationary}, if for every $s, t, k \in \mathbb{Z}$, the value of the expectations $\expect{\xb(t+k)\xb^{T}(s+k)}$ and $\expect{\xb(t+k)}$ is independent of $k$.
%We denote the $n \times n$ identity matrix by $I_n$.

\subsection{Admissible scheduling and wide-sense stationarity (ZWSII) w.r.t. scheduling}
\label{Subsec:assumptions1}
 Below we define the concept of admissible input processes and ZWSII processes w.r.t. scheduling. These concepts will be used to define the class of aLPV-SSAs
 we will work with. 
 %assume that the scheduling process $\p$ is an admissible process in the terminology of \cite{PetreczkyBilinear}.
 To this end, we will need the following notation from automata theory which will be used for other 
 purposes too
\begin{Notation}[Sequences over $\Sigma$]
	Consider the finite set
	\[ \Sigma = \{1, \ldots, \pdim\}. \]%and letter $\sigSet$ will denote an integer. 
%\end{Notation}
A \emph{non empty word} over $\Sigma$ is a finite sequence of letters, i.e., $w = \sigma_{1}\sigma_{2}\cdots \sigma_{k}$, where $0 < k \in \mathbb{Z}$, $\sigma_{1}, \sigma_{2}, \ldots, \sigma_{k} \in \Sigma$. The set of \emph{all} nonempty words is denoted by $\Sigma^{+}$. We denote an \emph{empty word} by $\epsilon$. Let $\Sigma^{*} = \epsilon \cup \Sigma^{+}$. The concatenation of two nonempty words $v = a_{1}a_{2}\cdots a_{m}$ and  $w= b_{1}b_{2}\cdots b_{n}$ is defined as $vw = a_{1}\cdots a_{m} b_{1} \cdots b_{n}$ for some $m,n > 0$. Note that if $w = \epsilon$ or $v= \epsilon$,  then $v\epsilon = v$ and $\epsilon w = w$, moreover, $\epsilon \epsilon = \epsilon$. The length of the word $w \in \Sigma^{*}$ is denoted by $|w|$, and $|\epsilon| =0$. 
\end{Notation}
%\begin{Example*}
%	For $\pdim=2$,  $\Sigma = \{1,2\}$, $\Sigma^{*}= \{\epsilon,1,2,11,12,21,22,111, \ldots\}$, for the word $w=111 \in \Sigma^{*}$, $|w|=3$.
%\end{Example*}
%Below we will state a number of technical assumptions. We will start by stating some assumptions on the scheduling process $\p$. More precisely, we will have to impose assumptions on the product of the scheduling variables. To this end, we define the following notation.

With the notation above, we can identify the process $\p$ with the collection 
$\{\p_{\sigma}\}_{\sigma \in \Sigma}$ of its components. 
We will say that $\p$ is an \emph{admissible scheduling process}, if the collection
$\{\p_{\sigma}\}_{\sigma \in \Sigma}$ is an admissible collection of input processes in the sense of
\cite[Definition 1]{PetreczkyBilinear} for $S=\Sigma \times \Sigma$. For the convenience of the reader, a version of \cite[Definition 1]{PetreczkyBilinear} is presented in Appendix \ref{App:def}.
Before proceeding further we will present examples of admissible scheduling sequences.
\begin{Example}[White noise  scheduling]
\label{i.i.d. input}
 The scheduling process $\p = [\p_1, \p_{2}, \ldots, \p_{\pdim}]^{T}$ is independent identically distributed (i.i.d.) such that
 for all $i,j=2,\ldots, \pdim$, $t \in \mathbb{Z}$, $\p_i(t), \p_j(t)$ are independent and
$\p_i(t)$ is zero mean, then $\p$ is admissible. %In this case, $\p$ is a admissible scheduling process %$p_{\sigma}=E[\p_{\sigma}^2(t)]$. 
% It to see that 
% for all $t \in \mathbb{Z}$, we have $\uw{1}(t) \equiv 1$, and for each $\sigma =2, \ldots, \pdim$, $\uw{\sigma}$ is zero-mean, and $p_1=1$ and $\pi_i=
% independent identically distributed (i.i.d.) process.
 \end{Example}
 \begin{Example}[Discrete valued i.i.d process]
 \label{disc:input}
  Assume there exists an i.i.d process $\btheta$
  which takes its values from a finite set $\Sigma=\{1,\ldots,\pdim\}$.  Let $\p_{\sigma}(t)=\chi(\btheta(t)=\sigma)$ for all $\sigma \in \Sigma$, $t \in \mathbb{Z}$. %Let $S=\Sigma \times \Sigma$, $L=\Sigma^{+}$ and  
  %$p_{\sigma}=P(\btheta(t)=\sigma)$, $\alpha_{\sigma}=1$ for all $\sigma \in \Sigma$. 
  Then $\p(t)=\begin{bmatrix} \p_1(t) & \ldots & \p_{n_{\mathrm p}} \end{bmatrix}^T$ is an admissible scheduling processes. 
 \end{Example}
 For further examples, see \cite{PetreczkyBilinear}. 
 
 We assume that the scheduling process $\p$ is admissible throughout the paper. 

Furthermore, we will use the definition of a \emph{Zero Mean Wide Sense Stationary (abbreviated by  ZMWSII) process with respect to $\p$} from \cite[Definition 2]{PetreczkyBilinear}.
%Similarly, we will say that 
%$\r$ is \emph{square integrable  (abbreviated by SII) with respect to $\p$}, if
%$\r$ is SII w.r.t. $\{\p_{\sigma}\}_{\sigma \in \Sigma}$
%according to \cite[Definition 5]{PetreczkyBilinear}.
%In addition to the assumptions on the scheduling process, we present assumptions on the input and output processes w.r.t. the scheduling process $\p$. First, it is necessary to mention that the
For the convenience of the reader this definition is reformulated in Appendix \ref{App:def} as Definition \ref{def:ZMWSSI}.
%and Definition \ref{Defn:SII} of 
%Appendix \ref{App:def}.

In order to explain the intuition behind these definitions, 
and because we will use them later on,
we define the following products of scheduling variables along a sequence from $\Sigma^{+}$.
For every word $\wordSet{+}$ where $w=\word$, $k \geq 1$, $\sigma_{1},\ldots, \sigma_{k} \in \Sigma$, we define the  process $\uw{w}$ as follows:
%the scalar $p_w \in \mathbb{R}$ as  follows
\begin{equation}
%	\begin{align}
\label{eqn:uw}
		\uw{w}(t) = \uw{\sigma_{1}}(t-k+1)\uw{\sigma_{2}}(t-k+2)\cdots\uw{\sigma_{k}}(t)  %\forall t \in \mathbb{Z} \nonumber \\
	%	p_w&=p_{\sigma_1}p_{\sigma_2} \cdots p_{\sigma_k}.
%	\end{align}
\end{equation}
For an empty word $w= \epsilon$, we set $\uw{\epsilon}(t)=1$. % $p_{\epsilon}=1$.
%It is in our interests for the covariances of $\p_i(t),\p_j(t)$ to be independent of past values of $\p$ and for $\uw{w}(t)$, $\uw{v}$ to be essentially orthogonal for $w \ne v$, unless $w$ and $v$ are suffixes of one another. 
%
%Therefore, we will be using \emph{admissible scheduling sequences}. Note that, for a detailed explanation, the definition of  is 

If $\p$ is admissible, then $\uw{w},\uw{v}$ are jointly wide-sense stationary. Moreover, $\uw{w}(t)$ and $\uw{v}(t)$ are uncorrelated, if the last letters of $w$ and
$v$ are different, and 
$$E[\uw{w\sigma}^2(t)]=p_{\sigma} E[\uw{w}^2(t)]$$ 
for any $w \in \Sigma^{+}$, $\sigma \in \Sigma$. 
That is, $\p$ uniquely determines a collection of numbers $\{p_{\sigma}\}_{\sigma \in \Sigma}$. This latter collection will play an important role in the sequel. 
In particular, for $\p$ from Example \ref{i.i.d. input}, 
$p_i$ is the variance of $\p_i$, and for $\p$ from Example \ref{disc:input},
$p_i$ is the probability  $P(\theta(t)=i)$, for all $i \in \Sigma$. 

In order to explain the significance of 
these assumptions, let $\r$ be a stochastic ZMWSII process w.r.t. $\p$,
and because we will use them latter on, we define the following products. 
% \subsection{Processes which are wide-stationary with respect to the scheduling}
%\label{Subsec:assumptions2}
%Next, we will discuss an extension of the notion of wide-sense stationarity. To this end, assume that
%$\p$ is admissible.
Let $\{p_{\sigma}\}_{\sigma \in \Sigma}$  be the constants determined by $\p$ as explained above, and
define the products
\begin{equation}
	\begin{aligned}\label{eqn:uw1}
%		\uw{w}(t) &= \textcolor{red}{\uw{\sigma_{1}}(t-k+1)\uw{\sigma_{2}}(t-k+2)\cdots\uw{\sigma_{k}}(t) } \\ %\forall t \in \mathbb{Z} \nonumber \\
		p_w&=p_{\sigma_1}p_{\sigma_2} \cdots p_{\sigma_k}.
	\end{aligned}
\end{equation}
For an empty word $w= \epsilon$, we set $p_{\epsilon}=1$.
For a stochastic process $\mathbf{r} \in \mathbb{R}^{\udim}$ and for each $\wordSet{*}$ we define the stochastic process $\zwr$ as
\begin{equation}\label{eqn:zwu}
	\zwr(t) = \mathbf{r}(t-|w|) \uw{w}(t-1)\frac{1}{\sqrt{p_{w}}},
\end{equation}
where $\uw{w}$ and $p_w$ are as in \eqref{eqn:uw} and \eqref{eqn:uw1}. For $w=\epsilon$, $\zwr(t)=\r(t)$.
The process $\zwr$ in \eqref{eqn:zwu} is interpreted as the product of the \emph{past} of $\r$ and $\p$.
The process $\zwr$ will be used as predictors for future values of $\r$ for various choices of $\r$. 

%\emph{Zero Mean Wide Sense Stationary with respect to inputs} process and \emph{Square integrable} process respectively, are used throughout the paper. The two definitions can be found in Appendix \ref{App:def}.
Now we will explain the motivation for the concept of
ZMWSII. Assume that  $\r$ is ZMWSII w.r.t. $\p$. Then $\zwr(t)$ is wide-stationary and 
square-integrable for all $w \in \Sigma^{+}$. Moreover,
the covariances $\expect{\zwr(t)(\zvr(t))^{T}}$ do not depend on $t$. Furthermore
$\zwr(t)$, $\zvr(t)$ are orthogonal, if $w$ is not a suffix of $v$ or vice versa.
Recall that we say that $w$ is a suffix of $v$, if $v=sw$ for some $s \in \Sigma^{*}$.
Moreover, when $w$ is suffix of $v$, then 
%In particular, from \cite[Lemma ]{PetreczkyBilinear} it follows that
%if $\mathbf{r}$ is ZMWSII w.r.t. $\p$, then
\begin{align*}
		& \expect{\zwr(t)(\zvr(t))^{T}} = \left\{ \begin{array}{ll} %\expect{\zsr(t)(\r(t))^{T}} & \text{if } w = sv \\
		\expect{\r(t)(\zsr(t))^{T}} & \text{if } v = sw \\
	\expect{\zsigr(t)(\zsigr(t))^{T}} & \text{if } v=w = \sigma s \\
		 %& \text{if neither} 
		\end{array} \right. \\
%		& \expect{\zwsigr(t)(\zvsigpr(t))^{T}} = \left\{ \begin{array}{ll}		%
%		\expect{\zwr(t)(\zvr(t))^{T}} & \hspace{-0.7em} \text{if } \sigma = \sigma \\
%		0 & \hspace{-0.7em} \text{if } \sigma \neq \sigma' \\
		%\end{array} \right. \\
%		& \text{If $|w| = |v|$ then the following definition can be applied} \\
%		& \expect{\zwr(t)(\zvr(t))^{T}} = \left\{ \begin{array}{ll} 
%		\expect{\zsigr(t)(\zsigr(t))^{T}} & \text{if } v=w = \sigma s \\
%		0 & \text{if } w \neq v \\
		%\end{array} \right.
\end{align*}
That is, $\expect{\zwr(t)(\zvr(t))^{T}}$ depends only on the difference
of $w$ and $v$, i.e., on the prefix of $v$.  
This can be viewed as a generalization of wide-sense stationarity,
if the index $w$ and $v$ are viewed as additional multidimensional time instances, and
$\Sigma^{*}$ is viewed as an additional time-axis.

%All the processes considered in this paper will be assumed to be ZMWSSI and SII process w.r.t. the scheduling process $\p$.  

%Using the concept of ZMWSSI process w.r.t. scheduling $\p$, we can formulate the main assumption regarding the process $(\mathbf{y},\p)$.
%\begin{Assumption}%[Assumption on the input and output process]
%	\label{asm:main}%
%	$\p$ is consider an admissible scheduling sequence (satisfies Definition \ref{asm:A1}), and 
%	$\mathbf{y}$ is a ZMWSSI and SII process w.r.t. $\p$ (satisfies Definition \ref{def:ZMWSSI} and Definition \ref{Defn:SII}).  
%\end{Assumption}

%Next, in the following section, we recall from \cite[Definition 2]{MejariCDC2019} the notion of a \emph{stationary} stochastic LPV-SSA representation of a process $\r$ without inputs.
\subsection{Stationary LPV-SSA representation}
\label{Subsec:assumptions3}
After identifying the necessary process properties and notations, we are now ready to present the definition of the \emph{stationary autonoumous stochastic LPV-SSA}.
\begin{Definition}\label{defn:LPV_SSA_wo_u}
	A \emph{stationary autonomous stochastic LPV-SSA}, abbreviated as \emph{asLPV-SSA}, 
%a process $\r$ taking values in $\mathbb{R}^{p}$, 
is a system of the form \eqref{eq:aslpv}, such that:
      
	\textbf{(1)}
		$\begin{bmatrix} \xb^T & \vb^T \end{bmatrix}^T$  is a ZMWSSI process, and 
		for all $\sigSet$, $ w \in \Sigma^{+}$,
        $E[\z^{\xb}_{\sigma}(t)(\z^{\vb}_{\sigma}(t))^T]=0, ~ E[\vb(t)(\z^{\xb}_w(t))^T]=0.$

	\textbf{(2)}
		%$\vb$ is a white noise process w.r.t. $\p$ \cite[Definition 3]{MejariLPVS2019}, i.e., 
		$\vb$ is ZMWSSI and $E[\vb(t)(\z^{\vb}(t))^T]=0$ for all $w \in \Sigma^{+}$.
		%$E[\z_{\sigma w}^{\vb}(t)(\z_{\sigma w}^{\r}(t))^T]=E[\z^{\vb}_{\sigma}(t)(\z^{\vb}_{\sigma}(t))^T] > 0$, for all $w \in \Sigma^{+}$, $\sigma \in \Sigma$.
	
	\textbf{(3)}
		The eigenvalues of the matrix $\sum_{\sigSet} \psig A_{\sigma} \otimes A_{\sigma}$ are inside the open unit circle.
		%	\end{enumerate} 
	%We call $\xb$ the state process and $\vb$ the noise process. 
       %We say that a asLPV-SSA is a realization of a process $\tilde{\r}$, if its output process equals $\r$. 
\end{Definition}
Note that, condition \textbf{(2)} implies that v is a white noise.

In the terminology of \cite{PetreczkyBilinear}, an asLPV-SSA corresponds to a stationary \GBS\  w.r.t. inputs $\{\p_{\sigma}\}_{\sigma \in \Sigma}$. 
%From \cite{PetreczkyBilinear}, if a process $\r$ can be represented by an asLPV-SSA, then $\r$ is a ZMWSSI process and 
%$\xb$ is uniquely determined by $\vb$ and the matrices $( %\{A_{\sigma},K_{\sigma}\}_{\sigma \in \Sigma}, C,D)$. 
Note that the processes $\xb$ and $\yb$ are ZMWSII, in particular, they
are wide-sense stationary, and that $\xb$ is orthogonal to the future values of the noise process $\vb$. We should mention that we concentrate on wide-sense stationary processes, because it is difficult to estimate the distribution of non-stationary processes. Also, wide-sense stationary processes solve the problem of the initial state conditions.

The state of an asLPV-SSA
is uniquely determined by its matrices and noise process. In order to
present this relationship, we need the following notation.
%In order to define this notion more precisely, 
\begin{Notation}[Matrix Product]\label{not:product}
	Consider $n \times n$ square matrices $\{A_{\sigma}\}_{\sigSet}$. For any word $\wordSet{+}$ of the form $w = \sigma_{1}\sigma_{2}\cdots\sigma_{k}$, $k\!>\!0$ and $\sigma_{1}, \ldots, \sigma_{k} \in \Sigma$, we define 
	%\begin{equation*}
	$$A_{w} = A_{\sigma_{k}}A_{\sigma_{k-1}}\cdots A_{\sigma_{1}}$$.
	%\end{equation*}
	%	For example, $w = 123$, $A_{w} = A_{123} = A_{3}A_{2}A_{1}$.
	For an empty word $\epsilon$, let $A_{\epsilon} = I_n$.
\end{Notation}
From \cite[Lemma 2]{PetreczkyBilinear} it follows that
% From \cite{PetreczkyBilinear}, it follows that
\begin{equation*}
	%\label{stat:state:eq1}
	%\begin{split}
	 \xb(t)=\sum_{v \in \Sigma^{*},\sigma \in \Sigma}
	   \sqrt{p_{\sigma v}} A_{v}K_{\sigma}\z^{\vb}_{\sigma v}(t)
	%
	%\sum_{\sigma_0 \in \Sigma} \sqrt{p_{\sigma_0}} K_{\sigma_0}\z^{\vb}_{\sigma_0}(t) +   
	\\ 
	%& \sum_{k=1}^{\infty} \sum_{\sigma_0,\sigma_1,\ldots, \sigma_k \in \Sigma} 
	% \sqrt{p_{\sigma_0\sigma_1 \cdots \sigma_k}}  A_{\sigma_k}A_{\sigma_{k-1}} \cdots A_{\sigma_1}K_{\sigma_0}\z^{\vb}_{\sigma_0 \cdots \sigma_k}(t) 
	% \end{split}
\end{equation*}
where the infinite sum on the right-hand side is  absolutely convergent in the mean square sense.
This prompts us to use the following notation.
\begin{Notation}
       We identify the asLPV-SSA $\mathcal{S}$ of the form \eqref{eq:aslpv} with the tuple 
       $\mathcal{S} =(\{A_{\sigma},K_{\sigma}\}_{\sigma=1}^{\pdim},C,F,\vb)$.
\end{Notation}

Finally, we need to define what we mean by an asLPV-SSA realization of a process.
An asLPV-SSA $\mathcal{S}$ of the form \eqref{eq:aslpv} is a \emph{realization of
a pair $(\tilde{\y},\tilde{\p})$}, if $\tilde{\y}=\y$, $\tilde{\p}=\p$. 
If $\mathcal{S}$ is of the form \eqref{eq:aslpv}, then we 
call the state-space dimension $\nx$ the \emph{dimension} of $\mathcal{S}$ and we denote it by
$\dim \mathcal{S}$.
We say that the asLPV-SSA $\mathcal{S}$ is a \emph{minimal} realization of $(\tilde{\y},\tilde{\p})$, if for any asLPV-SSA realization $\mathcal{S}^{'}$ of
$(\tilde{\y},\tilde{\p})$, the dimension of $\mathcal{S}^{'}$ is not smaller than $\dim \mathcal{S}$.

\section{Existence and minilality of asLPV-SSAs in innovation form}
\label{sect:min}
In this section we review the principal results on
existence and minimality of asLPV-SSA in innovation form. To this end, in  Subsection \ref{sect:real:red}  we recall from \cite{PetreczkyLPVSS} some results on realization theory of deterministic LPV-SSA. 
In Subsection \ref{sect:def:innov} we present the definition of asLPV-SSAs in innovation form, and
in Subsection \ref{subsect:min} we present results on existence and uniqueness of minimal asLPV-SSAs.
In Subsection \ref{Sect:algo} we present rank conditions for minimality of asLPV-SSAs and an algorithm for converting
any asLPV-SSA to a minimal one in innovation form. 
The results presented in this section follow from \cite{PetreczkyBilinear}, 
but they have never been formulated
%In contrast to \cite{PetreczkyBilinear}, we formulate the results below 
explicitly for LPV-SSAs. 

\subsection{Deterministic LPV-SSA representation}
\label{sect:real:red}
%In order to present the results of this paper, it is necessary to introduce the deterministic LPV-SSA.
%which will be used in the process of obtaining an asLPV-SSA representation in innovation form. 
%For this reason, we 
Recall from \cite{PetreczkyLPVSS,CoxLPVSS} that a deterministic LPV state-space representation with affine dependence (abbreviated as \emph{dLPV-SSA})
is a system of the form:
\begin{equation}
	\label{eqn:LPV_SSA:det} 
	\mathscr{S}\left\{\begin{aligned}
		&\mathrm{x}(t+1)=\sum_{i=1}^{\pdim} (\mathcal{A}_i\mathrm{x}(t)+\mathcal{B}_i\mathrm{u}(t))\mu_i(t), ~ \\
		&\mathrm{y}(t)=\mathcal{C}\mathrm{x}(t)+\mathcal{D}\mathrm{u}(t),
	\end{aligned}\right.
\end{equation}
where $\mathcal{A}_i, \mathcal{B}_i, \mathcal{C}, \mathcal{D}$ are matrices of suitable dimensions,  
$\mathrm{x}:\mathbb{Z} \rightarrow \mathbb{R}^{n_\mathrm{x}}$ is the state trajectory
$\mathrm{u}:\mathbb{Z} \rightarrow \mathbb{R}^{n_\mathrm{u}}$ is the input trajectory
$\mathrm{y}:\mathbb{Z} \rightarrow \mathbb{R}^{n_\mathrm{y}}$ is the output trajectory with finite support.
% In order to avoid technical problems, 
%we assume that $x,u,y$ all have finite support, i.e. there exist a $t_0 \in \mathbb{Z}$, such that
%x(s)=0,y(s)=0,u(s)=0$ for all $s < t_0$. 
We identify a dLPV-SSA of the form \eqref{eqn:LPV_SSA:det} with the tuple 
\begin{equation} 
\label{eqn:LPV_SSA:det_not} 
\mathscr{S}=(\{\mathcal{A}_{i},\mathcal{B}_{i}\}_{i=0}^{\pdim},\mathcal{C},\mathcal{D})
\end{equation}
%The number $\nx$ is called the dimension of $\mathscr{S}$. 

%\textcolor{red}{Manas: please make the notation consistent with tha of stochastic LPV-SSA, i.e. use $(\{A_{i},B_{i}\}_{i=0}^{\pdim},C,D)$ instead.} 
The \emph{sub-Markov function} of the dLPV-SSA  $\mathscr{S}$ 
is the function $M_{\mathscr{S}}:\Sigma^{*} \rightarrow \mathbb{R}^{n_\mathrm{y} \times n_\mathrm{u}}$, such
that for all $w \in \Sigma^{*}$, 
\begin{equation}\label{eqn:sub_markov}
	M_{\mathscr{S}}(w)=\left\{\begin{array}{ll}
		\mathcal{CA}_s\mathcal{B}_{\sigma}, \ \ & w=\sigma s, \, \sigma \in \Sigma, \, s \in \Sigma^{*} \\
		\mathcal{D}.           \ \  & w=\epsilon
	\end{array}\right.
\end{equation}
The values of $\{M_{\mathscr{S}}(w)\}_{w \in \Sigma^{*}}$ are the \emph{sub-Markov parameters} of $\mathscr{S}$.
From \cite{PetreczkyLPVSS} it then follows that two dLPV-SSAs have the same input-output behavior, if and only if
their sub-Markov functions are equal. 
%Moreover, the values of the sub-Markov function (sub-Markov parameters) can be
%determined from the input-output behavior.
%Sub-Markov parameters are analogous to Markov-parameters of linear systems, and just like in the linear case, a Ho-Kalman-like realization algorithm can be formulated which computes a dLPV-SSA from sub-Markov parameters \cite{PetreczkyLPVSS,RolandAbbas}. 
For a function $M:\Sigma^{*} \rightarrow \mathbb{R}^{n_\mathrm{y} \times n_\mathrm{u}}$, we will say that the dLPV-SSA $\mathscr{S}$ is a \emph{realization} of
$M$, if $M$ equals the sub-Markov function of $\mathscr{S}$, i.e., $M=M_{\mathscr{S}}$.
For dLPV-SSA \eqref{eqn:LPV_SSA:det}, we call the integer $\nx$ the \emph{dimension} of $\mathscr{S}$.
We will call a dLPV-SSA \emph{minimal}, if there exists no other dLPV-SSA $\mathscr{S}^{'}$ with a smaller dimension and with the
same sub-Markov function.
%$\nx$ is larger than the dimension of $\mathscr{S}^{'}$ and their sub-Markov parameters are equal, i.e. $M_{\mathscr{S}}=M_{\mathscr{S}}^{'}$. 
We call a dLPV-SSA $\mathscr{S}$ a \emph{minimal  realization} of a function $M:\Sigma^{*} \rightarrow \mathbb{R}^{n_\mathrm{y} \times n_\mathrm{w}}$, if 
$\mathscr{S}$ is minimal and it is a realization of $M$. 
From \cite{PetreczkyLPVSS}, it follows that a dLPV-SSA is minimal if and only if it is span-reachable and observable, and the latter
properties are equivalent to rank conditions of the extended $n$-step reachability and observability matrices \cite[Definition 1, Theorem 2]{PetreczkyLPVSS}.
Furthermore, any dLPV-SSA  can be transformed to a minimal one  with the same sub-Markov function, using Kalman decomposition \cite[Corollary 1]{PetreczkyLPVSS}\footnote{\cite[Corollary 1]{PetreczkyLPVSS} should be applied with zero initial state}.
For a more detailed discussion see \cite{PetreczkyLPVSS}. %Note that by \cite{PetreczkyLPVSS}, minimal dLPV-SSAs %with the same sub-Markov functions are isomorphic: if
%$\mathscr{S}$ from \eqref{eqn:LPV_SSA:det} and
%$\mathscr{S}^{'}=(\{\mathcal{A}_{\sigma}^{'},\mathcal{B}_{\sigma}^{'}\}_{\sigma \in \Sigma},\mathcal{C}^{'},\mathcal{D}^{'})$ are both minimal dLPV-SSA and $M_{\mathscr{S}}=M_{\mathscr{S}^{'}}$, $\mathcal{D}=\mathcal{D}^{'}$, then, 
%there exists a non-singular matrix $\mathcal{T}$ such that $\mathcal{TA}_{\sigma}\mathcal{T}^{-1}=\mathcal{A}_{\sigma}^{'}$, $\mathcal{TB}_{\sigma}=\mathcal{B}_{\sigma}^{'}$ and $\mathcal{C}^{'}=\mathcal{CT}^{-1}$. 

\subsection{Definition of asLPV-SSA in innovation form}
\label{sect:def:innov}
Next, we  define what we mean by asLPV-SSA in innovation form.
%and we show that if $(\y,\p)$ has a realization by an asLPV-SSA, it has a minimal realization in innovation form. However, before we define a realization in innovation form,
To this end, we need to introduce the following notation for orthogonal projection.
\begin{Notation}[Orthogonal projection $E_l$]
	\label{hilbert:notation}
	Recall that the set of square integrable
	random variables taking values in $\mathbb{R}$, forms a Hilbert-space with the scalar product defined as $<\mathbf{z}_1,\mathbf{z}_2>=E[\mathbf{z}_1\mathbf{z}_2]$. We denote this Hilbert-space by $\mathcal{H}_1$. 
	Let $\mathbf{z}$ be a square integrable \emph{vector-valued} 
	random variable taking its values in $\mathbb{R}^k$.  Let $M$ be a closed subspace   of $\mathcal{H}_1$. 
	By the orthogonal projection of $\mathbf{z}$ onto the subspace $M$, \emph{denoted by $E_l[\mathbf{z} \mid M]$},
	we mean the vector-valued square-integrable random variable $\mathbf{z}^{*}=\begin{bmatrix} \mathbf{z}_1^{*},\ldots,\mathbf{z}_k^{*} \end{bmatrix}^T$ such that $\mathbf{z}_i^{*} \in M$ is the orthogonal projection of the $i$th coordinate $\mathbf{z}_i$ of $\mathbf{z}$ onto $M$, as it is usually defined for Hilbert spaces. 
	Let $\mathfrak{S}$ be a subset of square integrable random variables in $\mathbb{R}^p$ for some integer $p$, and 
	suppose that $M$ is generated by the coordinates of the elements of $\mathfrak{S}$, i.e. 
	$M$ is the smallest (with respect to set inclusion) closed subspace of $\mathcal{H}_1$
	which contains the set $\{ \alpha^Ts \mid  s \in \mathfrak{S}, \alpha \in \mathbb{R}^p\}$. 
	Then instead of $E_l[z \mid M]$ we use \( E_{l}[\mathbf{z} \mid \mathfrak{S}] \). %to denote the projection of $z$ to $M$.
	%Notice that $\mathbf{z}^{*}$ is uniquely determined by the following property: $E[(\mathbf{z}-\mathbf{z}^{*})\x^T]=0$ for all $\x \in Z$.
	% That is, $M$ is the smallest (with respect to set inclusion) closed linear subset such that for any $\mathbf{z}=(\mathbf{z}_1,\ldots,\mathbf{z}_p) \in Z$, $M$ contains the scalar random variables $\mathbf{z}_1,\ldots,\mathbf{z}_p$.
	%Notice also that one can interpret $E_{l}[\mathbf{z} \mid Z]$ as the best approximation (prediction) 
	%of $\mathbf{z}$ in terms of (infinite) linear combinations of elements~of~$Z$. 
\end{Notation}
This said, we define the \emph{innovation process $\eb$ of $\yb$ with respect to
$\p$} as follows:
	\begin{equation} 
		\label{decomp:lemma:innov2}
		\eb(t)=\yb(t)-E_l[\yb(t) \mid \{\z^{\yb}_w (t)\}_{w \in \Sigma^{+}}]
	\end{equation}
In other words, $\eb(t)$ is the difference between the output and its projection on its past values w.r.t. the scheduling process $\p$, i.e., $\eb$ is the best predictor of $\yb$ using the product of the output and scheduling past values from \eqref{eqn:zwu}.
\begin{Definition}[asLPV-SSA in innovation form]
    An asLPV-SSA of the form \eqref{eq:aslpv} is said to be in \emph{innovation form}, if it is a realization of $(\y,\p)$, $F=I_{\ny}$,  and $\vb$ is the innovation process of $\yb$, i.e., $\vb = \eb$.
\end{Definition}

\subsection{Existence and uniqueness of minimal asLPV-SSA in innovation form}
\label{subsect:min}
Let  $\mathcal{S}$ %=(\{A_i,K_i,\}_{i=0}^{\pdim},C,I_{\ny},\eb)$ 
be an asLPV-SSA of the form \eqref{eq:aslpv} in innovation form with $F=I_{\ny}$.
%$,i.e., $\mathcal{S}$ is an asLPV-SSA of $(\yb,\p)$ in innovation form.
%(\{A_i,K_i,B_i\}_{i=0}^{\pdim},C,D,\eb^s)$ and  
Let $\widetilde{\mathcal{S}}=(\{\widetilde{A}_i,\widetilde{K}_i,\}_{i=0}^{\pdim},\widetilde{C},I_{\ny},\eb)$
be another asLPV-SSA of $(\yb,\p)$ in innovation form.
We say that $\mathcal{S}$ and $\widetilde{S}$ are \emph{isomorphic}, if there exists a nonsingular matrix $T$ such that
\begin{equation*}
\label{lemma:min1:eq1}
\widetilde{A}_i=TA_iT^{-1}, \widetilde{K}_i=TK_i, \widetilde{C}=CT^{-1}
\end{equation*}
We will say that the process $(\yb,\p)$ is \emph{full rank}, if for all $i=1,\ldots, \pdim$, $Q_i=E[\eb(t)(\eb(t))^T\p_i^2(t)]$ is invertable. This is a direct extension of the classical notion of a full rank process. 

Furthermore, we will say
that $\yb$ is \emph{Square Integrable} process w.r.t. $\p$, 
abbreviated by \emph{SII}, if
it satisfies \cite[Definition 5]{PetreczkyBilinear}\footnote{with $\bu_{\sigma}=\p_{\sigma}$,
$\sigma \in \Sigma$ in the terminology of \cite{PetreczkyBilinear}}.
For the convenience of the reader, the definition of an SII process is presented in Appendix \ref{App:def}, Definition \ref{Defn:SII}. From \cite[Remark 2]{PetreczkyBilinear} it follows that if $\yb,\p)$ has a realization by an asLPV-SSA and $\p$ is bounded, then $\yb$ is SII.
\begin{Theorem}[Existence and uniqueness]
\label{theo:min}
 Assume that $(\yb,\p)$ has an asLPV-SSA  realization and that $\yb$ is full rank and SII w.r.t. $\p$.
It follows that:	
\begin{enumerate}
\item
\label{theo:min:1}
$(\y,\p)$ has a minimal asLPV-SSA realization in innovation form
\item
\label{theo:min:2}
Any two minimal asLPV-SSA realizations of $(\yb,\p)$ in innovation form are isomorphic.  
\end{enumerate}
\end{Theorem}

The proof of this theorem follows \cite[Theorem 2]{PetreczkyBilinear} by
using the correspendence between \GBS\ and asLPV-SSA. 
%in the terminology of \cite{PetreczkyBilinear}, an asLPV-SSA corresponds to a stationary generalized bilinear system w.r.t. $\p$.

Theorem \ref{theo:min} implies that asLPV-SSAs in innovation form have the useful property that, when they are minimal, they are unique up to isomorphism, and
assuming that the asLPV-SSA which generates the data is minimal does not restrict the class of 
possible output.
That is, for system identificaion it is preferable to consider parameterizing the elements of minimal asLPV-SSAs in innovation form. This motivates finding conditions for an asLPV-SSA to be minimal in innovation form and formulating algorithms for transforming an asLPV-SSA to a minimal one
in innovation form.  To this end, in Subsection \ref{Sect:algo} we present a rank condition 
for minimality and a minimization algorithm based on the results of \cite{PetreczkyBilinear}. 
However, the results of Subsection \ref{Sect:algo} do not allow checking that an asLPV-SSA is in innovation form.
Moreover, the algebraic conditions for minimality are difficult to apply.
%Moreover, the rank conditions and the minimization algorithms require
%the solution of a Lyapunov-style equality which makes the difficult to apply to parametrizations. 
Motivated by this, in Section \ref{sect:main} we present more user-friendly 
characterizations of minimality and being in innovation form.

%\subsection{Conditions for minimality of asLPV-SSA, minimization and realization algorithms}
%  Minimality of asLPV-SSAs can be partially characterized via rank conditions, similar to
%  rank conditions for deterministic LPV-SSA. Moreover, it is possible to formulate an 
%  algorithm for transforming an asLPV-SSA to a minimal one in innovation form. 
  
 \subsection{Rank conditions and minimization algorithm}
 %for transforming a asLPV-SSA to minimal asLPV-SSA in innovation form}
\label{Sect:algo}
%To be able to obtain an asLPV-SSA representation in the forwrad innovation form, we present a reliable algorithm defined in the following definitions, where we associate a dLPV-SSA with the original asLPV-SSA and, afterwards, using the later dLPV-SSA, we associate with it the asLPV-SSA in innovation form.
%\begin{Definition}[dLPV-SSA associated with asLPV-SSA]\label{defn:det_asco_stationary}
In order to present the rank conditions for minimality of asLPV-SSA and the minimization algorithm, we need
to define the \emph{dLPV-SSA associated with asLPV-SSA} as
$$\mathscr{S}_{\mathcal{S}}=(\{\sqrt{p_i} A_i,\Bs_i\}_{i=1}^{\pdim},C,I_{\ny}), $$
where  %$\{\Bs_i\}_{i=1}^{\pdim}$ %are expressed as follows:
\begin{equation}
\label{min:eq-1}
\Bs_i=\frac{1}{\sqrt{p_i}}(A_iP_iC^T + K_iQ_iF^T)
\end{equation}
and  $Q_i=E[\vb(t)\vb^T(t)\p_i^2(t)]$ for $i=1,\ldots,\pdim$, 
and $\{P_i\}_{i=1}^{\pdim}$ are the computed as follows:
%\begin{Remark}[Computing $\Bs_i$]
%\label{gi:comp}
% In order to compute  $\{P_i\}_{i=1}^{\pdim}$ and hence $\{\Bs_i\}_{i=1}^{\pdim}$, we can use
%\cite[Lemma 5]{PetreczkyBilinear}, which states:
 %\begin{equation}
%\label{gi:comp:eq}
 \(    P_i = \lim_{\mathcal{I} \rightarrow \infty} P_i^{\mathcal{I}} \),
 %\end{equation}
 and $\{P_i^N\}_{i=1}^{\pdim}$ satisfy the following recursions:
\begin{equation}
\label{gi:comp:eq1}
    P_i^{\mathcal{I}+1}=p_{i} \sum_{j=0}^{\pdim} (A_jP_j^{\mathcal{I}}A_j^T + K_jQ_jK_j^T)
\end{equation}
with $P_i^0 = \mathbf{O}_{n_x,n_x}$ where $\mathbf{O}_{n_x,n_x}$ denotes the $n_x \times n_x$ matrix with all zero entries.
  %In particular, $P_i=p_{i} \sum_{j=0}^{\pdim} (A_jP_jA_j^T + K_jQ_jK_j^T)$ . 
%\end{Remark}
%\end{Definition}
Note that the existence of the limit of $P_i^N$ when $N$ goes to infinity follows from \cite[Lemma 5]{PetreczkyBilinear}.
%We will call \eqref{eq:aslpv} \emph{span-reachable} and \emph{observable}, if the associated 
%dLPV-SSA $\mathscr{S}_{\mathcal{S}}$ is span-reachable, respectively observable.
From \cite{PetreczkyBilinear} it follows that
the dLPV-SSA associated with an asLPV-SSA $\mathscr{S}$ represents a realization of Markov-function
$\covseq: \Sigma^{*} \rightarrow \mathbb{R}^{\ny \times \ny}$, 
$$\covseq(w) =\left\{ \begin{array}{ll} E[\yb(t) (\zwy(t))^{T} ] & w \in \Sigma^{+}\\
I_{\ny} & w=\epsilon \\ \end{array}\right.$$ 
computed from covariances of $\y$.

Then from \cite{PetreczkyBilinear} we can derive the following.
\begin{Theorem}[Rank conditions]
\label{theo:rank_cond}
 An asLPV-SSA is a minimal realization of $(\y,\p)$, if and only if the associated dLPV-SSA is minimal.
\end{Theorem}
From \cite{PetreczkyLPVSS} it follows that minimality of the associated dLPV-SSA can be checked using rank
conditions for the corresponding extended reachability and observability matrices. 
Note, however, that minimal asLPV-SSAs, realizing the same output, may not be isomorphic. In fact, in Section \ref{sect:examp}, Example \ref{sim:examp} presents a counter-example. 

%We can not only associate a dLPV-SSA  with an asLPV-SSA which realizes the sequence $\covseq$, 
Vice versa, with any
dLPV-SSA realization of $\covseq$ we can associate an asLPV-SSA in innovation form. This latter relationship
is useful for formulating a minimization algorithm. 
%\begin{Definition}[asLPV-SSA associated with dLPV-SSA]\label{defn:det_asco_stationary}
More precisely, consider a dLPV-SSA 
$$\mathscr{S}=(\{\hat{A}_i, \hat{G}_i\}_{i=1}^{\pdim}, \hat{C}, I_{\ny})$$
which is a minimal realization
of $\covseq$. 
%$\sum_{i=1}^{\pdim} \hat{A}_i \otimes \hat{A}_i$ is stable,
Define the \emph{asLPV-SSA $\mathcal{S}_{\mathscr{S}}$ associated with $\mathscr{S}$} as 
$$ \mathcal{S}_{\mathscr{S}}=(\{\frac{1}{\sqrt{p_i}} \hat{A}_i, \hat{K}_i\}_{i=1}^{\pdim},\hat{C},I_{\ny},\eb), $$
where $\hat{K}_i=\lim_{\mathcal{I} \rightarrow \infty} \hat{K}_i^{\mathcal{I}}$, and
	%\begin{equation*}
	   %\hat{A}^{s}_{\sigma}=\frac{1}{\sqrt{p_{\sigma}}} \hat{A}_{\sigma}, ~ 
       %$$\hat{K}_{\sigma}=\lim_{i \rightarrow \infty} \hat{K}_{\sigma}^{\mathcal{I}}$$
	%\end{equation*} 
	$\{\hat{K}_{\sigma}^{\mathcal{I}} \}_{\sigma \in \Sigma, \mathcal{I} \in \mathbb{N}}$ satisfies the following recursion
	\begin{equation}
		\label{statecov:iter}
		\begin{split}
			%\hat{P}^{i+1}_{\sigma} \!&=\! \sum \limits_{\sigma_{1} \in \Sigma, \sigma_{1} \sigma \notin L } \psig \left( \hat{A}_{\sigma_{1}} \hat{P}^{i}_{\sigma_{1} }  \hat{A}^{T}_{\sigma_{1}}  + \hat{K}_{\sigma_{1}} \hat{Q}^{i}_{\sigma_{1} }  \hat{K}^{\mathcal{I}})^{T}_{\sigma_{1}}  \right)\\
			& 	\hat{P}^{\mathcal{I}+1}_{\sigma} = \sum \limits_{\sigma_{1} \in \Sigma} \psig \left(\frac{1}{p_{\sigma_1}} 
			\hat{A}_{\sigma_{1}} \hat{P}^{\mathcal{I}}_{\sigma_{1} }  (\hat{A}_{\sigma_{1}})^{T}  + \hat{K}_{\sigma_{1}}^{\mathcal{I}} \hat{Q}^{\mathcal{I}}_{\sigma_{1} }  (\hat{K}^{\mathcal{I}}_{\sigma_{1}})^T  \right)\\
			& \hat{Q}^{\mathcal{I}}_{\sigma} = \psig T_{\sigma,\sigma}^{y} - \hat{C}  \hat{P}^{\mathcal{I}}_{\sigma } (\hat{C})^{T}  \\ 
			& \hat{K}^{\mathcal{I}}_{\sigma} = \left( \hat{\Bs}_{\sigma} \sqrt{\psig} - \frac{1}{\sqrt{\psig}}
			\hat{A}_{\sigma} \hat{P}^{\mathcal{I}}_{\sigma}  (\hat{C})^{T} \right) \left(  \hat{Q}^{\mathcal{I}}_{\sigma} \right)^{-1}
		\end{split}
	\end{equation}
	where $\hat{P}^{0}_{\sigma}$ is a $\nx \times \nx$ zero matrix and $T_{\sigma,\sigma}^{y} = \expect{\y(t) (\zwy(t))^{T}}$. 
%Intuitively, the limit of the matrices \eqref{statecov:iter}
%µsolve an extension of the well-known algebraic Riccati equation,
%see \cite[Remark 7]{PetreczkyBilinear}.
%\end{Definition}
 %\begin{Lemma}
%\label{min:col}
From \cite{PetreczkyBilinear} it follows that
%If $\mathcal{S}$ is an asLPV-SSA representation of $(\yb,\p)$, then the associated dLPV-SSA $\mathscr{S}_{\mathcal{S}}$ is a %realization of $M_{\yb}$, where $M_{\yb}$ is the Markov-parameters which is equal to the following covariance matrix:
%\begin{equation}\label{eqn:covseq}
%	M_\y = \covseq(w) = %\left\{\begin{array}{rl}
%	E[\yb(t) (\zwy(t))^{T} ]  %& w \in \Sigma^{+} \\
%	%I_{\ny} &  w=\epsilon 
	%             \end{array}\right. 
%\end{equation}
%Conversely, if $(\yb,\p)$ has an asLPV-SSA, and 
if $\mathscr{S}$ is a minimal dLPV-SSA realization of $\covseq$, then the associated asLPV-SSA $\mathcal{S}_{\mathscr{S}}$ is an asLPV-SSA of $(\yb,\p)$. 
%\end{Lemma}
%We will call \eqref{eq:aslpv} \emph{span-reachable} and \emph{observable}, if the associated 
%dLPV-SSA $\mathscr{S}_{\mathcal{S}}$ is span-reachable, respectively observable.
%the extended $\nx-1$-step reachability and observability matrices of $\mathscr{S}_{\mathcal{S}}$ satisfy $\rank \mathcal{R}_{\nx-1}=\nx$ and $\rank \mathcal{O}_{\nx-1}$.
%We will say that an asLPV-SSA of the form \eqref{eq:aslpv} is in innovation form, if the noise process $\vb$ equals the innovation process $\eb$ from  
%\eqref{decomp:lemma:innov2}. Note that the asLPV-SSA associated with a dLPV-SSA is in %innovation form. 

The discussion above suggests the following realization algorithm
for transforming an asLPV-SSA $\mathcal{S}$ to a minimal one, which can be deduced from \cite{PetreczkyBilinear}.
\begin{algorithm}[h!]
\caption{Minimization algorithm
\label{alg:min}}
~~\textbf{Input :} The asLPV-SSA $\mathcal{S} =(\{A_{\sigma},K_{\sigma}\}_{\sigma=1}^{\pdim},C,F,\vb)$ representation matrices.
\hrule\vspace*{.1cm}
\begin{enumerate}
    \item  Compute the dLPV-SSA $\mathscr{S}_{\mathcal{S}}$ associated with $\mathcal{S}$ and compute $T_{\sigma,\sigma}^{\y} = \frac{1}{\psig}(C\Psig C^{T} + F\Qsig F^T)$
           	%Note that, as a consequence for  $\mathcal{S}_{\mathscr{S}}$ being an asLPV-SSA representation of $(\y,\p)$ of the %form \eqref{eq:aslpv}, we can expressed the covariance as:

    \item Transform  the dLPV-SSA $\mathscr{S}_{\mathcal{S}}$ to
    a minimal dLPV-SSA 
    $\mathscr{S}_m$ using \cite[Corollary 1]{PetreczkyLPVSS}.
    
    \item Construct the asLPV-SSA $\mathcal{S}_{{\mathscr{S}}_m}$
    associated with $\mathscr{S}_m$.
\end{enumerate}
\hrule\vspace*{.1cm}
\textbf{Output :} The asLPV-SSA $\mathcal{S}_{{\mathscr{S}}_m}$.
%in \eqref{gi:comp:eq}
%(for computing $\mathscr{S}_{\mathcal{S}}$) and 
% \eqref{statecov:iter} (for computing $\mathcal{S}_{\mathscr{S}}$)  the usual numerical solutions are used to
%approximate limits of converging sequences.
\end{algorithm} 

From \cite[Theorem 3]{PetreczkyBilinear}, it follows that Algorithm \ref{alg:min} returns a minimal realization of $(\yb,\p)$, if $\mathcal{S}$ is an asLPV-SSA realization of $(\yb,\p)$ in innovation form. Note that Algorithm \ref{alg:min} requires only the knowledge of the matrices of $\mathcal{S}$ and the noise covariance matrix $Q_\sigma=E[\vb(t)\vb^T(t)\p_\sigma^2(t)]$. Also note that all the steps above are computationally efficient, however, they require finding the limits
of $P_{\sigma}^{\mathcal{I}}$ and $\hat{K}_{\sigma}^{\mathcal{I}}$ for $\mathcal{I} \rightarrow \infty$ respectively. 
Also note that, in \cite{MejariCDC2019}, it exists another minimization algorithm which uses covariances matrices.
\begin{Remark}[Challenges]
\label{rem:motiv1}
The main disadvantage of  
verifying the rank condition of Theorem \ref{theo:rank_cond} or applying Algorithm \ref{alg:min} is the necessity of constructing a dLPV-SSA and 
the necessity to find the limit of the matrices in
\eqref{statecov:iter}. The latter represents an extension of
algebraic Riccati equations \cite[Remark 7]{PetreczkyBilinear} and even for
the linear case requires attention. Moreover, the rank
conditions of Theorem \ref{theo:rank_cond} are not easy to apply to
parametrizations: even if the dependence of the matrices $\{A_i,K_i\}_{i=1}^{\pdim}$
and $C$ on a parameter $\theta$ are linear or polynomial, the dependence of
the matrices of the associated dLPV-SSA need not remain linear or polynomial,
due to the definition of $\Bs_i$ in \eqref{min:eq-1}. For the same reason,
it is difficult to analyze the result of 
applying Algorithm \ref{alg:min} to elements of a parametrization.
Moreover, the
conditions of Theorem \ref{theo:rank_cond} do not allow us to check if
the elements of a parametrizations are in innovation form. 
These shortcomings motivate the contribution of Section \ref{sect:main}.
\end{Remark}
%Therefore, in the following section we propose new conditions for an asLPV-SSA in innovation form.

 %In  this section, we present an algebraic conditions, in order to verify the existence of a realization in innovation form. %To this end, in Subsection \ref{sect:real:red} we recall the definition of a deterministic LPV-SSA representation, before presenting the main conditions in Subsection \ref{sect:main:cond}. 

%\subsection{Necessary conditions for a representation in innovation form}
%\label{sect:main:cond} 
\section{Main results: algebraic conditions for an asLPV-SSA to be minimal in innovation form}
\label{sect:main}

Motivated by the challenges explained in Remark \ref{rem:motiv1}, 
in this section we present sufficient conditions for an asLPV-SSA to be minimal and in innovation form. 
These conditions depend only on the matrices of the asLPV-SSA in question and do not require any information on the
noise processes.

  The first result concerns an algebraic characterization of asLPV-SSA
  in innovation form. This characterization does not require any knowledge
  of the noise process, only the knowledge of system matrices. 
  %it requires only checking the eigenvalues of a 
  %suitable matrix, constructed from the system matrices of an asLPV-SSA.
  In order to streamline the discussion, we introduce the following definition.
  \begin{Definition}[Stably invertable w.r.t. $\p$]
   Assume that $\mathcal{S}$ is an asLPV-SSA of the form \eqref{eq:aslpv} and  $F=I_{\ny}$.
  We will call $\mathcal{S}$ \emph{stably invertable with respect to $\p$}, or \emph{stably invertable} if $\p$ is clear from the context, if 
  the matrix
   \begin{equation}
    \label{inv:gbs:lemma:eq1}
    \sum_{i=1}^{\pdim} p_i (A_i-K_iC) \otimes (A_i-K_iC)
   \end{equation} 
   is stable (all its eigenvalues are inside the complex unit disk).
  \end{Definition}
   Note that a system can be stably invertable w.r.t. one scheduling process, and not to be
   stably invertable w.r.t. another one. 
   We can now state the result relating stable invertability to asLPV-SSAs in innovation forms.
  \begin{Theorem}[Innovation form condition]
  \label{inv:gbs:lemma}
    Assume that $\y$ is SII and $(\y,\p)$ is full rank.
    If an asLPV-SSA realization of $(\y,\p)$ is stably invertable, then it is in innovation form.
    %and assume that  Then $\vb$ is the innovation process of $(\y,\p)$ and $\mathcal{S}$ is a %realization of $(\y,\p)$ in innovation form.
  \end{Theorem}
  The proof of Theorem \ref{inv:gbs:lemma} can be found in Appendix \ref{App:proof}.
  
  Stably invertable asLPV-SSAs can be viewed as optimal predictors.
  %form as optimal predictors. 
  Indeed, let $\mathcal{S}$ be the asLPV-SSA of the form
\eqref{eq:aslpv} which is in innovation form, and let $\x$ be the
 state process of $\mathcal{S}$.  
 It then follows 
	\begin{equation}
       \label{gen:filt:bil:def:pred}
       \begin{split}
      & \x(t+1) = \sum_{i=1}^{\pdim}  (A_i-K_iC)\x(t)+K_i\y(t))\p_i(t), \\
	& \hat{\y}(t) = C\x(t)
	\end{split}
	\end{equation}
where $\hat{\y}(t)=E_l[\y(t) \mid \{\z_w^{\y}(t)\}_{w \in \Sigma^{+}}]$, i.e.,
$\hat{\y}$ is the best linear prediction of $\y(t)$ based on 
the predictors $\{\z_w^{\y}(t)\}_{w \in \Sigma^{+}}$. 
Intuitively, \eqref{gen:filt:bil:def:pred} could be viewed as a filter, i.e.,
a dynamical system
driven by past values of $\y$ and generating the best possible
linear prediction $\hat{\y}(t)$ of 
$\y(t)$ 
%Moreover, the variance of the prediction error $\e(t)=\y(t)-\hat{\y}(t)$
%is the smallest possible among all the linear predictions of $\y(t)$ 
based on $\{\z_w^{\y}(t)\}_{w \in \Sigma^{+}}$.
However, the solution of \eqref{gen:filt:bil:def:pred} is defined on the whole time
axis $\mathbb{Z}$ and hence cannot be computed exactly. For stably
invertable asLPV-SSA we can approximate $\hat{\y}(t)$ as follows.
\begin{Lemma}
\label{gbs:finite_filt:lemma}
 With the assumptions of Theorem \ref{inv:gbs:lemma}, if $\mathcal{S}$ of the form \eqref{eq:aslpv} is
 a stably invertable realization of $(\y,\p)$, and we consider the following dynamical system:
  \begin{equation}
  \label{gbs:finite_filt:eq}
   \begin{split}
   & \bar{\x}(t+1)= \sum_{i=1}^{\pdim} (A_i-K_iC)\bar{\x}(t)+K_i\y(t))\bmu_i(t), \\
  &    \bar{\y}(t) = C\bar{x}(t), ~ \bar{\x}(0)=0
   \end{split}
  \end{equation}
 then 
 \( \underset{t \rightarrow \infty}{\lim} \left(\bar{\x}(t) - \x(t)\right)\!\!=\!\!0 \), and
   \( \underset{t \rightarrow \infty}{\lim} \left(\bar{\y}(t) - \y(t)\right)=0 \),
   where the limits are understood in the mean square sense.
  \end{Lemma}
 The proof of Lemma \ref{gbs:finite_filt:lemma} is found in Appendix \ref{App:proof}.
 That is, the output
 $\bar{\y}(t)$ of the recursive filter
 \eqref{gbs:finite_filt:eq} is an approximation
 of the optimal prediction $\hat{\y}(t)$ of $\y(t)$ for large enough $t$. 
 That is, stably invertable asLPV-SSA not only result in asLPV-SSAs in innovation 
 form, but they represent a class of
 asLPV-SSAs for which recursive filters of
 the form \eqref{gbs:finite_filt:eq} exist.

  Next, we present algebraic conditions for minimality of an asLPV-SSA
  in innovation form. 
  %In contrast to Theorem \ref{theo:rank_cond}, the condition
  %below if formulated directly in terms of the system matrices.
  %and
  %it does not require the iterative solution of \eqref{gi:comp:eq1} and
  %the construction of the dLPV-SSA from . 
  \begin{Theorem}[Minimality condition in innovation form]
  \label{min:forw:gbs:lemma}
    Assume that $\mathcal{S}$ is an asLPV-SSA of the form \eqref{eq:aslpv} and that $\mathcal{S}$ is a realization of $(\y,\p)$ in innovation form.
    Assume that $(\y,\p)$ is full rank and $\y$ is SII.  
    Then $\mathcal{S}$ is a minimal realization of $(\y,\p)$, if and only if
    the dLPV-SSA $\mathcal{D}_{\mathcal{S}}=(\{A_i,K_i\}_{i=0}^{\pdim},C,I_{\ny})$ is minimal.
  \end{Theorem}
  The proof of Theorem \ref{min:forw:gbs:lemma} can be found in Appendix \ref{App:proof}.
  
  Theorem \ref{min:forw:gbs:lemma}, in combination with Theorem \ref{inv:gbs:lemma}, leads to the following corollary.
  \begin{Corollary}[Minimality and innovation form]
  \label{min:forw:gbs:lemma:col}
   With the assumptions of Theorem \ref{min:forw:gbs:lemma}, 
   if $\mathcal{D}_{\mathcal{S}}$ is minimal and $\mathcal{S}$ if stably invertable, then 
   $\mathcal{S}$ is a minimal asLPV-SSA realization of $(\y,\p)$
   in innovation form. 
   %If $\mathcal{S}$is an asLPV-SSA realization of $(\y,\p)$ of
    %the form \eqref{eq:aslpv}, 
    %$(\y,\p)$ is full rank, $\y$ is SII.
  \end{Corollary}
  %Note that the matrices of dLPV-SSA
  %$\mathcal{D}_{\mathcal{S}}$  are just the matrices of $\mathcal{S}$, hence 
  %applying Corollary \ref{min:forw:gbs:lemma:col} does not require any additional %computational
  %steps. 
  \begin{Remark}[Checking minimality and innovation form]
  \label{rem:check1}
  We recall that $\mathcal{D}_{\mathcal{S}}$ is minimal, if and only
  if it satisfies the rank conditions for the extended $n$-step reachability and observability matrices \cite[Theorem 2]{PetreczkyLPVSS},
  which can easily be computed from the matrices of $\mathcal{S}$.
  Checking that $\mathcal{S}$ is stably invertable boils down to checking the eigenvalues of the matrix \eqref{inv:gbs:lemma:eq1}.
  That is, Corollary \ref{min:forw:gbs:lemma:col} provides effective procedure for verifying that an asLPV-SSA is minimal and
  in innovation form.
  Note that in contrast to the rank condition of Theorem \ref{theo:rank_cond},
  which required computing the limit of \eqref{gi:comp:eq1}, the procedure
  above uses only the matrices of the system.
  \end{Remark}
  \begin{Remark}[Parametrizations of asLPV-SSAs]
   Below we will sketch some ideas for applying the above results to
   parametrizations of asLPV-SSAs. A detailed study of these issues remains a topic for future research.
   
   For all the elements of a parametrization of asLPV-SSAs to be minimal and in innovation form, by Corollary \ref{min:forw:gbs:lemma:col} it is necessary that
\textbf{(A)} all elements of the parametrization, when viewed as dLPV-SSA, are minimal, and that  \textbf{(B)} they
 are stably invertable and satisfy condition \textbf{(3)} of Definition \ref{defn:LPV_SSA_wo_u}.
 In order to deal with \textbf{(A)}, the techniques used in \cite{Alkhoury2016} or \cite{HSCC2010}
 \footnote{In order to use \cite{HSCC2010} the relationship between minimality of dLPV-SSA and that of switched systems \cite{PetreczkyLPVSS} should be epxloited}could be used.
 The condition  \textbf{(B)} is equivalent to stability  of a suitable parametrization of LTI systems, as we could view the matrices
 \eqref{inv:gbs:lemma:eq1} and $\sum_{q=1}^{\pdim} p_i A_i \otimes A_i$ as matrices
 of an LTI state-space representation.
 %and for \textbf{(B)} these matrices have to be stable in the sense of LTI systems.
  For the latter we can use standard techniques, see \cite{Ribarits2002} and the references therein. 
  
  Finally, note that extension of argument of \cite[Theorem 2]{Alkhoury2016}
  would also lead to identifiability conditions for parametrizations which
  satisfy the conditions \textbf{(A)} and \textbf{(B)} described above.
  \end{Remark}

  Corollary \ref{min:forw:gbs:lemma:col} suggests the following
  minimization algorithm. 
\begin{algorithm}[H]
\caption{Minimization algorithm
\label{alg:min1}
}
~~\textbf{Input :} A stably invertable asLPV-SSA $\mathcal{S}$
\hrule\vspace*{.1cm}
\begin{enumerate}
 %   \item  Compute the dLPV-SSA $\mathscr{S}_{\mathcal{S}}$ associated with $\mathcal{S}$ and compute
           	%Note that, as a consequence for  $\mathcal{S}_{\mathscr{S}}$ being an asLPV-SSA representation of $(\y,\p)$ of the %form \eqref{eq:aslpv}, we can expressed the covariance as:
%	\begin{equation*}
%    T_{\sigma,\sigma}^{\y} = \frac{1}{\psig}(C\Psig C^{T} + F\Qsig F^T)
%\end{equation*}
    \item Transform  the dLPV-SSA $\mathcal{D}_{\mathcal{S}}$ 
     from Theorem \ref{min:forw:gbs:lemma} to
    a minimal dLPV-SSA 
    $\mathcal{D}_m=(\{A_i^m,K_i^m\}_{i=1}^{\pdim},C^m,I_{\ny})$ using \cite[Corollary 1]{PetreczkyLPVSS}.
    
    %\item Construct the asLPV-SSA $\mathcal{S}_{{\mathscr{S}}_m}$
    %associated with $\mathscr{S}_m$.
\end{enumerate}
\hrule\vspace*{.1cm}
\textbf{Output :} asLPV-SSA $\mathcal{S}_m=(\{A_i^m,K_i^m\}_{i=1}^{\pdim},C^m,I_{\ny},\e)$. 
%in \eqref{gi:comp:eq}
%(for computing $\mathscr{S}_{\mathcal{S}}$) and 
% \eqref{statecov:iter} (for computing $\mathcal{S}_{\mathscr{S}}$)  the usual numerical solutions are used to
%approximate limits of converging sequences.
\end{algorithm} 
\begin{Lemma}[Correctness of Algorithm \ref{alg:min1}]
\label{alg:min1:lem}
The asLPV-SSA $\mathcal{S}_m$ is stably invertable and it is a minimal asLPV-SSA realization of
$(\y,\p)$ in innovation form. 
\end{Lemma}
That is, Algorithm \ref{alg:min1} is a simpler counterpart of
Algorithm \ref{alg:min} which does not require computing the limits
\eqref{gi:comp:eq1} and \eqref{statecov:iter}, which can create difficulties (see Remark \ref{rem:motiv1}).
%The latter limits are solution to
%extensions of Lyapunov and algebraic Riccation equations respectively \cite[Section MIII.C]{PetreczkyBilinear}.
Lemma \ref{alg:min1:lem} implies that stable invertability is preserved by minimization.
\section{Numerical examples}
\label{sect:examp}
In this section, we present numerical examples in order to illustrate the main results. 
%All computations are carried out on an i7 11th Gen Intel core processor at 2.5 GHz with 30 GB of RAM installed running MATLAB R2021b.

%We will present a couple of examples.
\begin{Example}
\label{sim:examp}
Consider an asLPV-SSA of the form \eqref{eq:aslpv},
where $\pdim = 2$ and
\begin{equation*}
\footnotesize
\begin{aligned}
    &A_1 = \begin{bmatrix}
    0.4 & 0.4 & 0 \\
    0.2 & 0.1 & 0 \\
    0 & 0 & 0.2 \\
    \end{bmatrix}, \quad
    K_1 = \begin{bmatrix}
    0 \\
    1 \\
    1 \\
    \end{bmatrix}, \quad
    C = \begin{bmatrix} 10 & 0 & 0 \\ \end{bmatrix} \\
    &A_2 = \begin{bmatrix}
    0.1 & 0.1 & 0 \\
    0.2 & 0.3 & 0 \\
    0 & 0 & 0.2 \\
    \end{bmatrix}, \quad
    K_2 = \begin{bmatrix}
    0 \\
    1 \\
    1 \\
    \end{bmatrix}, \quad \text{and} \quad F = 1 \\
    \end{aligned}
\end{equation*}
Note that this representation is not in innovation form.
%The representation corresponds to a state-dimension $n_x = 3$, output dimension $n_y = 1$, and scheduling dimension $n_\mu = 2$ so the finite set $\Sigma$ contains only two elements $(\Sigma = \{1,2\})$.
The scheduling signal process is defined as $\p = [\p_1 \quad \p_2]$ such that $\p_1(t) = 1$ and $\p_2(t)$ is a white-noise process with uniform distribution $\mathbf{\mathcal{U}}(-1.5,1.5)$. This corresponds to the parameters values $\{\psig\}_{\sigma \in \{1,2\}}$ to be $p_1 = \expect{\mu_1^2(t)} = 1$ and $p_2 = \expect{\mu_2^2(t)} = 0.75$.
The noise process is a white Gaussian noise with a variance equal to 1, i.e., $\vb \sim \mathbf{\mathcal{N}}(0,1)$.
Using Algorithm \ref{alg:min}, we can find a minimal representation in innovation form with the following matrices:
\begin{equation*}
\footnotesize
\begin{aligned}
    &A_1^m = \begin{bmatrix}
    0.4007 & 0.3997 \\
    0.1997 & 0.0993 \\
    \end{bmatrix}, \quad
    K_1^m = \begin{bmatrix}
    -0.046 \\
    -0.0541 \\
    \end{bmatrix} \\
    &A_2^m = \begin{bmatrix}
    0.1003 & 0.1002 \\
    0.2002 & 0.2997 \\
    \end{bmatrix}, \quad
    K_2^m = \begin{bmatrix}
    -0.0116 \\
    -0.0578 \\
    \end{bmatrix} \\
    &C^m = \begin{bmatrix} -10 & 0.0116 \\ \end{bmatrix}, \quad \text{and} \quad F^m = 1 \\
    \end{aligned}
\end{equation*}
%By verifying the stability of the matrix $\sum_{i=1}^{\pdim} p_i (A^m_i-K^m_iC^m) \otimes (A^m_i-K^m_iC^m)$, as Theorem \ref{inv:gbs:lemma} suggests, we conclude that the minimal representation is indeed in innovation form. 
Note that the two systems $\mathcal{S} =(\{A_{\sigma},K_{\sigma}\}_{\sigma=1}^{\pdim},C,F,\vb)$ and $\mathcal{S}^m =(\{A^m_{\sigma},K^m_{\sigma}\}_{\sigma=1}^{\pdim},C^m,F^m,\eb)$ have the same output trajectory when the chosen scheduling is applied. 
However, by changing the scheduling process to another process $\p^{'}$ with $\p_1^{'} = \p_1$ and $\p_2^{'}$ is also a white-noise with a uniform distribution $\mathbf{\mathcal{U}}(-\sqrt{3},\sqrt{3})$, we realize that the output trajectories are not the same, in other words, the two systems are not input-output equivalent. Subsequently, the two representations are not isomorphic.
\end{Example}

\begin{Example}
\label{sim:examp2}
%We present aa minimal asLPV-SSA which is not in innovation form. 
We use the same noise process $\vb$ and scheduling process $\p$ as in Example \ref{sim:examp}.
The system matrices are as follows.
\begin{equation*}
\footnotesize
\begin{aligned}
    &A_1 = \begin{bmatrix}
    0.4 & 0.4 \\
    0.2 & 0.1 \\
    \end{bmatrix}, \quad
    K_1 = \begin{bmatrix}
    0 \\
    1 \\
    \end{bmatrix}, \quad
    C = \begin{bmatrix} 10 & 0 \\ \end{bmatrix} \\
    &A_2 = \begin{bmatrix}
    0.1 & 0.1 \\
    0.2 & 0.3 \\
    \end{bmatrix}, \quad
    K_2 = \begin{bmatrix}
    0 \\
    1 \\
    \end{bmatrix},\quad 
    \text{and} \quad F = 1 \\
    \end{aligned}
\end{equation*}
The system above is not in innovation form, in fact, it is not stably invertable. 
As before, we use Algorithm \ref{alg:min}  to obtain a minimal asLPV-SSA  innovation form with the following matrices
%- after verifying the condition in Theorem \ref{inv:gbs:lemma}. The system has the the following representation matrices:  
\begin{equation*}
\footnotesize
\begin{aligned}
    &A_1^m = \begin{bmatrix}
    0.4007 & -0.3997 \\
    -0.1997 & 0.0993 \\
    \end{bmatrix}, \quad
    K_1^m = \begin{bmatrix}
    -0.046 \\
    0.0541 \\
    \end{bmatrix}, \\
    &A_2^m = \begin{bmatrix}
    0.1003 & -0.1002 \\
    -0.2002 & 0.2997 \\
    \end{bmatrix}, \quad
    K_2^m = \begin{bmatrix}
    -0.0116 \\
    0.0578 \\
    \end{bmatrix} \\
    &C^m = \begin{bmatrix} -10 & -0.0116 \\ \end{bmatrix}, \quad \text{and} \quad F^m = 1 \\
    \end{aligned}
\end{equation*}
The two asLPV-SSA systems $\mathcal{S} = (\{A_{\sigma},K_{\sigma}\}_{\sigma=1}^{\pdim},C,F,\vb)$ and $\mathcal{S}^m =(\{A^m_{\sigma},K^m_{\sigma}\}_{\sigma=1}^{\pdim},C^m,F^m,\eb)$ are not isomorphic. 
In fact, they have different output trajectories, when using the scheduling process $\p^{'}$ from Example \ref{sim:examp}. This is due to the fact that the system $\mathcal{S}$ is not in innovation form. 
\end{Example}
From Examples \ref{sim:examp} and \ref{sim:examp2}, we can conclude that minimality alone does not provide uniqueness.

\begin{Example}
\label{sim:examp3}
%This leads us to our last example, where 
We present a minimal asLPV-SSA in innovation form, where the state dimension $\nx = 2$ and its matrices are as follows:
\begin{equation*}
\footnotesize
\begin{aligned}
    &A_1 = \begin{bmatrix}
    0.4 & 0.4 \\
    0.2 & 0.1 \\
    \end{bmatrix}, \quad
    A_2 = \begin{bmatrix}
    0.1 & 0.1 \\
    0.2 & 0.3 \\
    \end{bmatrix} \\
    &K_1 = \begin{bmatrix}
    0 \\
    1 \\
    \end{bmatrix}, \quad
    K_2 = \begin{bmatrix}
    0 \\
    1 \\
    \end{bmatrix}, \quad
    C = \begin{bmatrix} 1 & 0 \\ \end{bmatrix}, \quad \text{and} \quad F = 1 \\
    \end{aligned}
\end{equation*}
The asLPV-SSA above is stably invertable, and hence in innovation form.
We use the same scheduling process $\p$ and noise process $\vb$  as Example \ref{sim:examp}.
If we apply Algorithm \ref{alg:min} to the system above, we get 
another asLPV-SSA with the matrices
%Despite $\mathcal{S} =(\{A_{\sigma},K_{\sigma}\}_{\sigma=1}^{\pdim},C,F,\vb)$ being a minimal system in innovation form - which can be verified using the matrix condition \eqref{inv:gbs:lemma:eq1} - the proposed algorithm will deliver another representation which is isomorphic to $\mathcal{S}$.
%\begin{Remark}
%The objective of the algorithm is to deliver a minimal representation in innovation form, for any representation with unknown properties. 
%However, in this example we aim to show the uniqueness of a minimal asLPV-SSA in innovation form.
%\end{Remark}
%In order to clear the paper's objective, the algorithm delivers the following matrices:
\begin{equation*}
\footnotesize
\begin{aligned}
    &A^m_1 = \begin{bmatrix}
    0.4642 & -0.3581 \\
    -0.1581 & 0.0358 \\
    \end{bmatrix}, \quad
    K^m_1 = \begin{bmatrix}
    -0.1143 \\
    0.9934 \\
    \end{bmatrix} \\
    &A^m_2 = \begin{bmatrix}
    0.1367 & -0.1188 \\
    -0.2188 & 0.2633 \\
    \end{bmatrix}, \quad
    K^m_2 = \begin{bmatrix}
    -0.1143 \\
    0.9934 \\
    \end{bmatrix} \\
    &C^m = \begin{bmatrix} -0.9934 & -0.1143 \\ \end{bmatrix}, \quad \text{and} \quad F^m = 1 \\
    \end{aligned}
\end{equation*}
As expected, the two systems are isomorphic, the corresponding matrix is
\begin{equation*}
\footnotesize
    T = \begin{bmatrix} -0.9934 & -0.1143 \\ -0.1143 & 0.9934 \\ \end{bmatrix}
\end{equation*}

Finally, we realize that the output trajectories of both systems are indeed the same not only for the chosen scheduling sequence, but also for \emph{any} other scheduling process.
\end{Example}
\section{Conclusion}
This paper formulates conditions for a LPV state-space representation  to be minimal and in innovation form.
These conditions depend only the matrices of the LPV
representation. A minimization algorithm for transforming any LPV representation to a minimal one in innovation form is formulated too.
%technical review on asLPV-SSA and presents a minimization algorithm in order to introduce new formal matrix conditions for an asLPV-SSA to be minimal and in innovation form. 
These results are expected to be useful for system identification, in particular, for making the identification problem mathematically well-posed.
In the future, we will explore the application of the proposed results to concrete system identification algorithms.
\bibliographystyle{plain}
\bibliography{Bib.bib}

\begin{thebibliography}{10}

\bibitem{Alkhoury2016}
Z.~Alkhoury, M.~Petreczky, and G.~Merc{\`e}re.
\newblock Structural properties of affine {LPV} to {LFR} transformation:
  minimality, input-output behavior and identifiability.
\newblock In {\em Proceedings of the IEEE Conference on Decision and Control},
  Las Vegas, USA, 2016.

\bibitem{bagi}
B.~A.\ Bamieh and L.\ Giarr\'e.
\newblock Identification of linear parameter-varying models.
\newblock {\em International Journal of Robust Nonlinear Control},
  12(9):841--853, 2002.

\bibitem{Bilingsley}
P.~Bilingsley.
\newblock {\em Probability and measure}.
\newblock Wiley, 1986.

\bibitem{CHIUSO2005377}
Alessandro Chiuso and Giorgio Picci.
\newblock Consistency analysis of some closed-loop subspace identification
  methods.
\newblock {\em Automatica}, 41(3):377--391, 2005.
\newblock Data-Based Modelling and System Identification.

\bibitem{CoxLPVSS}
P.~Cox, M.~Petreczky, and R.\ T{\'o}th.
\newblock Towards efficient maximum likelihood estimation of {LPV-SS} models.
\newblock {\em Automatica}, 97(9):392--403, 2018.

\bibitem{CoxTothSubspace}
Pepijn~Bastiaan Cox and Roland Tóth.
\newblock Linear parameter-varying subspace identification: A unified
  framework.
\newblock {\em Automatica}, 123:109296, 2021.

\bibitem{RamosSubspace}
P.L. dos Santos, J.A. Ramos, and J.L.M. de~Carvalho.
\newblock Identification of bilinear systems with white noise inputs: An
  iterative deterministic-stochastic subspace approach.
\newblock {\em IEEE Transactions on Control Systems Technology},
  17(5):1145--1153, Sept 2009.

\bibitem{FavoreelTAC}
W.~Favoreel, B.~De~Moor, and P.~Van~Overschee.
\newblock Subspace identification of bilinear systems subject to white inputs.
\newblock {\em IEEE Transactions on Automatic Control}, 44(6):1157--1165, Jun
  1999.

\bibitem{fewive}
F.\ Felici, J.~W.\ {Van Wingerden}, and M.\ Verhaegen.
\newblock Subspace identification of {MIMO} {LPV} systems using a periodic
  scheduling sequence.
\newblock {\em Automatica}, 43:1684--1697, 2007.

\bibitem{Katayama:05}
T.~Katayama.
\newblock {\em Subspace Methods for System Identification}.
\newblock Springer-Verlag, 2005.

\bibitem{TothKulcsar}
N.~Kulcs\'ar and T\'oth R.
\newblock On the similarity state transformation for linear parameter-varying
  systems.
\newblock In {\em Proc. 18th IFAC World Congress}, 2011.

\bibitem{laurainrefined}
V.~Laurain, M.~Gilson, R.~T{\'o}th, and H.~Garnier.
\newblock Refined instrumental variable methods for identification of {LPV
  Box--Jenkins} models.
\newblock {\em Automatica}, 46(6):959--967, 2010.

\bibitem{LindquistBook}
A.~Lindquist and G.~Picci.
\newblock {\em Linear Stochastic Systems: A Geometric Approach to Modeling,
  Estimation and Identification}.
\newblock Springe Berlin, 2015.

\bibitem{MejariLPVS2019}
M.~Mejari and M.~Petreczky.
\newblock {Realization and identification algorithm for stochastic LPV
  state-space models with exogenous inputs}.
\newblock In {\em {3rd IFAC Workshop on Linear Parameter-Varying Systems}},
  Eindhoven, Netherlands, 2019.

\bibitem{MEJARIAuto}
M.~Mejari, D.~Piga, and A.~Bemporad.
\newblock A bias-correction method for closed-loop identification of {L}inear
  {P}arameter-{V}arying systems.
\newblock {\em Automatica}, 87:128--141, 2018.

\bibitem{MejariCDC2019}
Manas Mejari and Mihaly Petreczky.
\newblock Consistent and computationally efficient estimation for stochastic
  lpv state-space models: realization based approach.
\newblock In {\em 2019 IEEE 58th Conference on Decision and Control (CDC)},
  pages 3805--3810. IEEE, 2019.

\bibitem{PetreczkyRealChapter}
M.~Petreczky.
\newblock Realization theory of linear hybrid systems.
\newblock In {\em Hybrid Dynamical Systems: Observation and control}, Lecture
  Notes in Control and Information Sciences. Springer-Verlag, 2015.

\bibitem{HSCC2010}
M.~Petreczky, L.~Bako, and J.H. van Schuppen.
\newblock Identifiability of discrete-time linear switched systems.
\newblock In {\em Hybrid Systems: Computation and Control}, pages 141--150.
  ACM, 2010.

\bibitem{PetreczkyLPVSS}
M.~Petreczky, R.\ T{\'o}th, and G.~Merc{\'e}re.
\newblock Realization theory for {LPV} state-space representations with affine
  dependence.
\newblock {\em IEEE Transactions on Automatic Control}, 62(9):4667--4674, 2017.

\bibitem{PetreczkyBilinear}
M.~Petreczky and R.~Vidal.
\newblock Realization theory for a class of stochastic bilinear systems.
\newblock {\em IEEE Transactions on Automatic Control}, 63(1):69--84, 2018.

\bibitem{PetreczkyNAHS}
M.~Petreczky, R.~Wisniewski, and J.~Leth.
\newblock Balanced truncation for linear switched systems.
\newblock {\em Nonlinear Analysis: Hybrid Systems}, 10:4--20, November 2013.

\bibitem{pigalpv}
D.\ Piga, P.\ Cox, R.\ T{\'o}th, and V.\ Laurain.
\newblock {LPV} system identification under noise corrupted scheduling and
  output signal observations.
\newblock {\em Automatica}, 53:329--338, 2015.

\bibitem{Ribarits2002}
T.~Ribarits.
\newblock {\em The role of parametrizations in identification of linear dynamic
  systems}.
\newblock PhD thesis, Vienna University of Technology, Vienna, Austria, 2002.

\bibitem{tanelliidentification}
M.\ Tanelli, D.\ Ardagna, and M.\ Lovera.
\newblock Identification of {LPV} state space models for autonomic web service
  systems.
\newblock {\em IEEE Transactions on Control Systems Technology}, 19(1):93--103,
  2011.

\bibitem{toth}
R.\ T{\'o}th.
\newblock {\em Modeling and identification of linear parameter-varying
  systems}.
\newblock Lecture Notes in Control and Information Sciences, Vol. 403,
  Springer, Heidelberg, 2010.

\bibitem{Wingerden09}
J.~W. {van Wingerden} and M.~Verhaegen.
\newblock Subspace identification of bilinear and {LPV} systems for open- and
  closed-loop data.
\newblock {\em Automatica}, 45(2):372--381, 2009.

\bibitem{Verdult02}
V.~Verdult and M.~Verhaegen.
\newblock Subspace identification of multivariable linear parameter-varying
  systems.
\newblock {\em Automatica}, 38(5):805--814, 2002.

\bibitem{veve}
V.\ Verdult and M.\ Verhaegen.
\newblock Kernel methods for subspace identification of multivariable {LPV} and
  bilinear systems.
\newblock {\em Automatica}, 41:1557--1565, 2005.

\end{thebibliography}
\appendix
\subsection{Technical definitions}
\label{App:def}
\begin{Definition}[Admissible scheduling sequences]

	\label{asm:A1}	
	%	\item[A1.] 
%	\textcolor{red}{The scheduling process $\p = [\p_1, \p_{2}, \ldots, \p_{\pdim}]^{T}$ is independent identically distributed (i.i.d.) such that, for all $t \in \mathbb{Z}$, we have $\uw{1}(t) \equiv 1$, and for each $\sigma =2, \ldots, \pdim$, $\uw{\sigma}$ is zero-mean.} % independent identically distributed (i.i.d.) process.
A scheduling process $\p$ is called an \emph{Admissible scheduling sequence} if it satisfies the following properties: 
\begin{enumerate}
 \item
 \label{gen:filt:ass_new:cond1:def1}
  Denote by $\mathscr{F}^{\p,-}_{t}$ the $\sigma$-algebra generated by the random variables $\{\p(k) \}_{k < t}$. There exists positive numbers
  $\{p_{\sigma}\}_{\sigma \in \Sigma}$ such that for any $w,v \in \Sigma^{+}, \sigma, \sigma^{'} \in \Sigma$, $t \in \mathbb{Z}$:
  \begin{equation}
  \label{gen:filt:ass_new:cond1:def1:eq1}
  \begin{split}
      & E[\uw{w\sigma}(t)\uw{v\sigma^{'}}(t) \mid \mathscr{F}^{\p,-}_{t} ] =  \\
      & \left\{\begin{array}{rl}
                                                                          p_{\sigma} \uw{w}(t-1)\uw{v}(t-1) & \sigma=\sigma^{'}  \\
                                                                          0 & \mbox{otherwise} 
                                                                          \end{array}\right. \\
     &  E[\uw{w\sigma}(t)\uw{\sigma^{'}}(t) \mid \mathscr{F}^{\p,-}_{t}] = \\
     &   \left\{\begin{array}{rl}
                                                                          p_{\sigma} \uw{w}(t-1)  & \sigma=\sigma^{'} \mbox{ and }  \\
                                                                          0 & \mbox{otherwise}
                                                                          \end{array}\right. \\
     & E[\uw{\sigma}(t)\uw{\sigma^{'}}(t) \mid \mathscr{F}^{\p,-}_{t}] = \begin{array}{rl}
                                                                          0 & \mbox{if } \sigma \ne \sigma^{'} 
                                                                          \end{array}
  \end{split}
  \end{equation}
\item
 \label{gen:filt:ass_new:cond2:def1}
 There exist real numbers $\{\alpha_{\sigma}\}_{\sigma \in \Sigma}$ such that 
 $\sum_{\sigma \in \Sigma} \alpha_{\sigma} \uw{\sigma}(t)=1$ for all $t \in \mathbb{Z}$. 

  \item
  For each $w,v \in \Sigma^{+}$, the process $\begin{bmatrix} \uw{w}, \uw{v} \end{bmatrix}^T$ is wide-sense stationary. 
\end{enumerate}

\end{Definition}
Definition \ref{asm:A1} is a reformulation of \cite[Definition 1]{PetreczkyBilinear}, for $S \subseteq \Sigma \times \Sigma$.

\begin{Definition}[ZMWSSI, {\cite[Definition 2]{PetreczkyBilinear}}]\label{def:ZMWSSI}
	A stochastic process $(\r,\p)$ is \emph{Zero Mean Wide Sense Stationary} (ZMWSSI) if 
	\begin{enumerate}
		\item For $\timeset$, the $\sigma$-algebras generated by the variables $\{\mathbf{r}(k) \}_{k \leq t}$, $\{\msig(k) \}_{k < t, \sigSet}$ and $\{\msig(k) \}_{k \geq t, \sigSet},$ denoted by $\filtr$, $\filt$ and $\filtp$ respectively, are such that $\filtr$ and $\filtp$ are conditionally independent w.r.t. $\filt$.
		%i.e., intuitively, future values of $\p$ do not directly depend on the past values of $\mathbf{r}$, however only via dependence of past $\mathbf{r}$ on past values of $\p$.
		\item The processes $\{\mathbf{r}, \{\zwr \}_{\wordSet{+}} \}$  are zero mean, square integrable and are jointly wide sense stationary, i.e., 
		$\forall t,s,k \in \mathbb{Z},w,v \in \Sigma^{+}$, 
		\begin{align*}
		&   \expect{\mathbf{r}(t)} = 0, ~  \expect{\zwr(t)} =0  \\
		&	\expect{\r(t+k)(\zwr(s+k))^{T}} = \expect{\r(t) (\zwr(s))^{T}}, \\
		&	\expect{\r(t+k)(\r(s+k))^{T}} = \expect{\r(t)(\r(s))^{T}}, \\
		&	\expect{\zwr(t+k)(\zvr(s+k))^{T}} = \expect{\zwr(t) (\zvr(s))^{T}}. 
				\end{align*}
	\end{enumerate}
\end{Definition} 

\begin{Definition}[SII process {\cite[Definition 5]{PetreczkyBilinear}}]%[Square Integrable w.r.t Inputs (SII) process]
\label{Defn:SII}
	A process $\r$ is said to be \emph{Square Integrable w.r.t. $\p$} (abbreviated as \emph{SII} when $\p$ is clear from the context), for all $\wordSet{*}, \timeset $,
	%the following process $\mathbf{z}^{\yb+}_w$ is square integrable.
	%	\begin{equation}
	the random variable $\mathbf{z}^{\r+}_w (t)= \mathbf{r}(t+|w|) \p_{w}(t+|w|-1)\frac{1}{\sqrt{p_w}}$,
	%\end{equation}
	is square integrable.
\end{Definition}

%%the following is a Comment

\subsection{Proofs of Theorems \ref{inv:gbs:lemma} and \ref{min:forw:gbs:lemma}}
\label{App:proof}

 First we show the following technical result, which states the following.
 \begin{Lemma}[Mean-square stability block-diagonal matrices]
 \label{proof:lemma-1}
   Consider matrices $\nx \times n$
   $\{B_{i=1}\}_{i=1}^{\pdim}$ and $n \times n$ matrices $\{F_{i}\}_{i=1}^{\pdim}$ and
 $\nx \times \nx$ matrices $\{A_i\}_{i=1}^{\pdim}$ such that 
  that $\sum_{i=1}^{\pdim} p_i (F_i \otimes F_i)$ and $\sum_{i=1}^{\pdim} p_i (A_i \otimes A_i$  are stable (all their eiegenvalues are inside the unit disk). 
  Conside the matrix
  \[ 
      \begin{split}
      &  \tilde{F}_i=\begin{bmatrix} F_i & B_i \\
                                  0   & A_i
                   \end{bmatrix} \\
      \end{split}              
    \]      
  $\sum_{i=1}^{\pdim} p_i (\tilde{F}_i \otimes \tilde{F}_i)$ is stable.  
 \end{Lemma}
 \begin{proof}
    Recall from \cite[Lemma 5]{PetreczkyNAHS} it follows that
    the stability of $\sum_{i=1}^{\pdim} p_i (\tilde{F}_i \otimes \tilde{F}_i)$
    is equivalent to the existence of a matrix $P > 0$ such that
    $P - \sum_{i=1}^{\pdim} p_i \tilde{F}_iP\tilde{F}_i^T > 0$.
    We will construct such a matrix $P$. To this end, 
    from \cite[Lemma 5]{PetreczkyNAHS} it follows that 
    there exist $P_2 > 0$ such that 
    $P_2-\sum_{i=1}^{\pdim} p_i A_iP_2A_i^T > 0$. Define
    \[ S=\sum_{i=1}^{\pdim} p_i(B_iP_2B_i^T) + \mathscr{D}(P_2-\sum_{i=1}{\pdim} p_i A_iP_2A_i^T)^{-1}\mathscr{D}^T \]
    where $\mathscr{D}=(\sum_{i=1}^{\pdim} p_i B_iP_2A_i^T$.
    It then follows that $S \ge 0$, as $(P_2-\sum_{i=1}{\pdim} p_i A_iP_2A_i^T)^{-1}$ is
    positive 
    definite and $B_iP_2B_i^T$ are postive semi-definite matrices for each $i=1,\ldots,\pdim$.
    From \cite[Lemma 5]{PetreczkyNAHS} it follows that there exist $P_1 > 0$ such that
    \begin{equation}
    \label{proof:lemma-1:eq1}
    P_1-\sum_{i=1}^{\pdim} F_iP_1F_i^T - S >  0 
    \end{equation}
    Consider $P=\begin{bmatrix} P_1 & 0 \\ 0 & P_2 \end{bmatrix}$.
    Using standard calculation, it follows that
    \[
      \begin{split}
       & P-\sum_{i=1}^{\pdim} p_i\tilde{F}_iP\tilde{F}_i^T = M(P) \\
      & M(P)= \\
      &\begin{bmatrix}
     P_1-\sum_{i=1}^{\pdim} p_i F_iP_1F_i^T - \mathcal{D}_1 & - \mathscr{D} \\
       -\mathscr{D}^T & P_2-\sum_{i=1}{\pdim} p_i A_iP_2A_i^T
       \end{bmatrix}
       \end{split}
    \]
    $\mathscr{D}=\sum_{i=1}^{\pdim} p_i B_iP_2A_i^T$, 
    $\mathcal{D}_1=\sum_{i=1}^{\pdim} p_i(B_iP_2B_i^T)$.
    Note that $P_2-\sum_{i=1}{\pdim} p_i A_iP_2A_i^T > 0$, hence, by using Schur complement,
    $M(P) > 0$ if  \eqref{proof:lemma-1:eq1} holds, which is true by the choice of $P_1$.
    \end{proof}
 \begin{Lemma}[Orthogonal noises]
 \label{proof:lemma-3}
  Let $\vb$ be a ZMWSII such that $E[\vb(t)\z^{\vb}_w(t)]=0$, $w \in \Sigma^{+}$. Then 
  $\z_w(t)^{\vb}(t), \z_v^{\vb}(t)$ are uncorrelated for all $w,v \in \Sigma^{+}$, $v \ne w$.
 \end{Lemma}
 \begin{proof}[Proof Lemma \ref{proof:lemma-3}]
  Clearly $E[\z_w^{\vb}(t)(\z_v^{\vb}(t))^T]$ is either zero if $w$ is not a suffix of $v$, or if
  $v$ is not a suffix of $w$, as it is ZMWSII.  If say $v=sw$, $s \ne \epsilon$ then 
  $E[\z_w^{\vb}(t)(\z_v^{\vb}(t))^T]=E[\vb(t)\z_s^{\vb}(t)]=0$. 
 \end{proof}
 \begin{Lemma}[Absolute convergence infinite sums]
 \label{proof:lemma-2}
 Consider a $p \times n$ matrix $C$ and matrices $n \times \ny$
   $\{B_{i=1}\}_{i=1}^{\pdim}$ and $n \times n$ matrices $\{F_{i}\}_{i=1}^{\pdim}$ such that 
  that $\sum_{i=1}^{\ny} p_i (F_i \otimes F_i)$ is stable. 
  Let $\vb$ be a ZMWSII such that $E[\vb(t)(\z^{\vb}_w(t))^T]=0$, $w \in \Sigma^{+}$.
  Then the infinite sum
  \begin{equation}
   \label{proof:lemma-2:eq1}
    \r(t)=\sum_{w \in \Sigma^{*},\sigma \in \Sigma} \sqrt{p_{\sigma w}} CF_wB_{\sigma}\z_{\sigma w}^{\vb}(t)
  \end{equation}
 is absolutely convergent in the mean-square sense, i.e., 
 \begin{equation}
   \label{proof:lemma-2:eq2}
    \sum_{w \in \Sigma^{*},\sigma \in \Sigma} E[\|CF_wB_{\sigma}\z_{\sigma w}^{\vb}(t)\|^2_2]
  \end{equation}
  and the process $\r(t)$ is ZMWSII and it the unique
  state process of the asLPV-SSA $(\{F_i,B_i\}_{i=1}^{\pdim},C,I,\vb)$.
 \end{Lemma}    
 \begin{proof}[Proof of Lemma \ref{proof:lemma-2}]
     From \cite[Lemma  3]{PetreczkyBilinear} it follows 
     that \eqref{proof:lemma-2:eq1} is convergent in mean square sense
     and $\r(t)$ is ZMWSII as it is the solution of the
     asLPV-SSA $(\{F_i,B_i\}_{i=1}^{\pdim},C,I,\vb)$.
     From Lemma \ref{proof:lemma-3} it follows that $\z_w(t)^{\vb}(t),\z_v^{\vb}(t)$ are 
     uncorrelated. Hence
     \[ 
     \begin{split}
     & E[\|\sum_{w \in \Sigma^{*},|w| \le N,\sigma \in \Sigma} CF_wB_{\sigma}\z_{\sigma w}^{\vb}(t)\|^2_2]= \\
     & trace \left(\sum_{w,v \in \Sigma^{*},|w|,|v| \le N,\sigma,\sigma_1 \in \Sigma}
       CA_wB_{\sigma}E[\z_w^{\vb}(t)(\z_v^{\vb}(t))^T](CA_vB_{\sigma_1})^T\right) = \\
     & \sum_{w \in \Sigma^{*}, |w| \le N, \sigma \in \Sigma} E[\|CF_wB_{\sigma}\z_{\sigma w}^{\vb}(t)\|^2_2]
     \end{split}
     \]
     and as $E[\|\sum_{w \in \Sigma^{*},|w| \le N,\sigma \in \Sigma} CF_wB_{\sigma}\z_{\sigma w}^{\vb}(t)\|^2_2]$ is convergent by mean-square convergence of \eqref{proof:lemma-2:eq1},
     it follows that \eqref{proof:lemma-2:eq2} is convergent. 
 \end{proof}
 \begin{Lemma}
 \label{proof:lemma1}
   Assume that $(\y,\p)$ has a asLPV-SSA realization. Consider matrices $n \times \ny$
   $\{B_{i=1}\}_{i=1}^{\pdim}$ and $n \times n$ matrices $\{F_{i}\}_{i=1}^{\pdim}$ such that 
  that $\sum_{i=1}^{\pdim} p_i (F_i \otimes F_i)$ is stable. Then the infinite sum
  \begin{equation}
   \label{proof:lemma1:eq1}
    \sum_{w \in \Sigma^{*},\sigma \in \Sigma} F_wB_{\sigma}\z_{\sigma w}^{\y}(t)
  \end{equation}
  converges absolutely in the mean-square sense and
  $\r(t)=\sum_{w \in \Sigma^{*},\sigma \in \Sigma} F_wB_{\sigma}\z_{\sigma w}^{\y}(t)$
  is a ZMWSII process, and $\r$ is the unique process $\bar{\r}$ which
  satisfies
  \begin{equation}
  \label{proof:lemma1:eq1.5}      
  \bar{\r}(t+1)=\sum_{i=1}^{\pdim} (F_i\bar{\r}(t)+B_i\y(t))\p_i(t)
  \end{equation}
  and $\begin{bmatrix} \bar{\r}^T & \y^T \end{bmatrix}$ is ZMWSII and the components of $\bar{\r}(t)$ belong to 
  the Hilbert-space $\mathcal{H}^{\y}_t$   generated by $\{\z_w^{\y}(t)\}_{w \in \Sigma^{+}}$.
\end{Lemma}
\begin{proof}[Proof of Lemma \ref{proof:lemma1}]
   Assume that $\mathcal{S}=(\{A_i,K_i\}_{i=1}^{\pdim},C,D,\vb)$ is a minimal asLPV-SSA 
   realization of $(\y,\p)$ in innovation form. Let us define the matrices
   \[ 
      \begin{split}
      &  \tilde{F}_i=\begin{bmatrix} F_i & B_iC \\
                                  0   & A_i
                   \end{bmatrix},  \\
     & \tilde{B}_i = \begin{bmatrix} B_iD \\ K_i \end{bmatrix} \\
     & \tilde{C} = \begin{bmatrix} I_{n} & 0 \end{bmatrix} 
      \end{split}              
    \]                
    From Lemma \ref{proof:lemma-1} it follows
    that $\sum_{i=1}^{\pdim} p_i (\tilde{F}_i \otimes \tilde{F}_i)$ is stable,
    and hence by Lemma \ref{proof:lemma-2}
    \begin{equation}
     \label{proof:lemma1:eq2}
     \tilde{\x}(t)=\sum_{w \in \Sigma^{*},\sigma \in \Sigma} \sqrt{p_\sigma w} \tilde{F}_w B_{\sigma} \z_{\sigma w}^{\vb}(t)
    \end{equation}
    is absolutely convergent in the mean square sense and it
    is the unique state process of $(\{\tilde{F}_i,\tilde{B}_i\}_{i=1}^{\pdim},\tilde{C},I,\vb)$.
    %Moreover, from the fact that
    %
    %it follows that the infinite sums
    %\[ 
    %   E(\|\sum_{w \in \Sigma^{*},\sigma \in \Sigma} \tilde{F}_w B_{\sigma} \z_{\sigma w}^{\v}(t)\|^2)=\|\sum_{w \in \Sigma^{*},\sigma \in \Sigma} E[\|tilde{F}_w B_{\sigma} \z_{\sigma w}^{\v}(t)\|^2]
    % \]
    %is convergent, i.e., \eqref{proof:lemma1:eq2} is absolutely convergent in the Hilbert-space
    %of square integrable random variables. 
    
    From \cite[Lemma 2 and Lemma 9]{PetreczkyBilinear} it follows that
    \begin{equation}
     \label{proof:lemma1:eq3}
     \z^{\y}_v(t)=\sum_{w \in \Sigma^{*},\sigma \in \Sigma} CA_w K_{\sigma} \z_{\sigma wv}^{\vb}(t) + D\z^{\vb}_{v}(t)
    \end{equation}
    Notice that
    \begin{equation}
    \label{proof:lemma1:eq4}
    \tilde{F}_{w}=\begin{bmatrix} F_w & \sum_{s_1,s_2 \in \Sigma^{*},\sigma_1 \in \Sigma, w=s_1\sigma s_2} F_{s_2}B_{\sigma_1}CA_{s_1}\\
     0 & A_w
     \end{bmatrix}
    \end{equation}
    and notice that
    $\hat{\x}(t)=\begin{bmatrix} \bar{\x}(t) \\  \x(t) \end{bmatrix}$,
    where $\x$ is the unique state process of the 
    asLPV-SSA $\mathcal{S}=(\{A_i,K_i\}_{i=1}^{\pdim},C,D,\vb)$
    which realizes $(\y,\p)$.
    In particular, $\bar{\x}(t)$
    \begin{equation}
    \label{proof:lemma1:eq5}
    \begin{split}
    & \bar{\x}(t)=\sum_{w \in \Sigma^{*},\sigma \in \Sigma} 
        \left( F_wB_{\sigma}D\z^{\vb}_{w\sigma}(t)  +  \right. \\ 
        & \left. \sum_{s_1,s_2 \in \Sigma^{*},\sigma_1 \in \Sigma, w=s_1\sigma s_2} F_{s_2}B_{\sigma_1}CA_{s_1}K_{\sigma} 
        \z_{\sigma w}^{\vb}(t)\right)
    \end{split}    
    \end{equation}
    From Lemma \ref{proof:lemma-2} it follows that 
    $S_3=\sum_{w \in \Sigma^{*},\sigma \in \Sigma} F_w B_{\sigma}D\z^{\vb}_{w\sigma}(t)$ is
    absolutely convergent in the mean-square sense. Hence, 
    \begin{equation}
       \label{proof:lemma1:eq6}
       \begin{split}
    & S_4:=\bar{\x}(t)-S_3= \\
    & =\sum_{w \in \Sigma^{*},\sigma_1,\in \Sigma} \left(\sum_{s_1,s_2 \in \Sigma^{*},\sigma_1 \in \Sigma, w=s_1\sigma s_2} F_{s_2}B_{\sigma_1}CA_{s_1}K_{\sigma} \z_{\sigma w}^{\vb}(t)\right) 
    \end{split}
    \end{equation}
    is absolutely convergent in the mean-square sense.
    It is known that for absolutely convergent series of elements Hilbert-spaces can be rearranged
    while preserving convergence, hence, by using \eqref{proof:lemma1:eq3}
    \begin{equation}
       \label{proof:lemma1:eq6}
     \begin{split}
    & S_4=\sum_{s_1,s_2\in \Sigma^{*} \sigma_1,\sigma \in \Sigma}
    \sqrt{p_{\sigma s_1\sigma_1s_2}} F_{s_2}B_{\sigma_1}CA_{s_1}K_{\sigma} \z_{\sigma w}^{\vb}(t) = \\
    & \sum_{s_2 \in \Sigma^{*},\sigma_1 \in \Sigma} F_{s_2}B_{\sigma_1}
      \underbrace{\left( \sum_{s_2 \in \Sigma^{*},\sigma \in \Sigma} CA_{s_2}K_{\sigma}\z_{\sigma s_2\sigma_1 s_1}^{\vb}(t)\right)}_{\z^{\y}_{\sigma_1 s_2}(t)-D\z_{\sigma_1 s_2}^{\vb}(t)} \\
    &   \sum_{s_2 \in \Sigma^{*},\sigma_1 \in \Sigma} F_{s_2}B_{\sigma_1}\z_{\sigma_1s_1}^{\y}(t)-F_{s_2} B_{\sigma_1}
    D\z_{\sigma_1 s_2}^{\vb}(t)=\\
    & \sum_{s_2 \in \Sigma^{*},\sigma_1 \in \Sigma} F_{s_2}B_{\sigma_1}\z_{\sigma_1s_1}^{\y}(t) - S_3
    \end{split}
    \end{equation}
    is absolutely convergent, and hence
    $\r(t)=\bar{\x}(t)=S_4+S_3$ is absolutely convergent. 
    Finally notice that $\r(t)=\bar{\x}(t)$ is the component of the
    unique state process $\tilde{\x}(t)$ of the asLPV-SSA
    $(\{\tilde{F}_i,\tilde{B}_i\}_{i=1}^{\pdim},\tilde{C},I,\vb)$
    and hence $\tilde{\x}$ is ZMWSII and hence so is
    $\r$. Finally, from $\tilde{\x}$ being the state process of 
    $(\{\tilde{F}_i,\tilde{B}_i\}_{i=1}^{\pdim},\tilde{C},I,\vb)$
    it follows that 
    $\r(t)$ satisfies \eqref{proof:lemma1:eq1.5}. Moreover,
    if $\bar{\r}(t)$ is a ZMWSII process which satisfies 
    \eqref{proof:lemma1:eq1.5} such that $\begin{bmatrix} \bar{\r}^T & \y^T \end{bmatrix}^T$ is ZMWSII and
    the components of $\bar{\r}(t)$ belong to $\mathcal{H}_t^{\y}$.
    Note that since $\mathcal{S}$ is a minimal asLPV-SSA realization o$(\y,\p)$ in innovation form,
    $\vb(t)=\y(t)-C\x(t)$ and the elements of $\vb$ belong to the Hilbert-space $\mathcal{H}^{\y}_{t}$. 
    Notice that from \eqref{proof:lemma1:eq3} it follows that
    $\mathcal{H}^{\y}_t$ is a subspace of the Hilbert-space $\mathcal{H}_t^{\vb}$ generated by $\{\z^{\vb}_{w}(t)\}_{w \in \Sigma^{+}}$.
    That is, $\mathcal{H}^{\y}_t=\mathcal{H}_t^{\vb}$. 
    Then $\hat{\x}=\begin{bmatrix} \bar{\r}^T & \x \end{bmatrix}^T$ is such that
    $\hat{\x}(t)$ satisfies the conditions of \cite[Lemma 10]{PetreczkyBilinear} for $\r=\vb$,
    $\begin{bmatrix} \hat{\x}^T & \vb^T \end{bmatrix}^T$ is ZMWSII. Moreover, $E[\vb(t)(\z_w^{\bar{\r}}(t))^T]=0$ as
    $\vb(t)$ is orthogonal to $\mathcal{H}_t^{\vb}=\mathcal{H}_t^{\y}$ and the components of $\z_w^{\bar{\r}}(t)$ belongs 
    $\mathcal{H}_t^{\y}$ for all $w \in \Sigma^{*}$. Finally $\hat{\x}(t+1)=\sum_{i=1}^{\pdim} (\tilde{F}_i\hat{\x}(t)+\tilde{B}_i\vb(t))\p_i(t)$. Hence, $\hat{\x}(t)$ is a state process of 
     $(\{\tilde{F}_i,\tilde{B}_i\}_{i=1}^{\pdim},\tilde{C},I,\vb)$ and hence it is unique and equals $\tilde{\x}$ and
     $\bar{\r}=\r$.

    %$\x_e(t)=\sum_{w \in \Sigma^{*},\sigma \in \Sigma} \tilde{F}_w B_{\sigma} \z_{\sigma w}^{\vb}(t)$ and the latter is
    %ZMWSII by Lemma \ref{proof:lemma-2}.
\end{proof}
\begin{proof}[Proof of Lemma \ref{gbs:finite_filt:lemma}]
 From Lemma \ref{proof:lemma1} it follows that 
 $\tilde{\x}(t)=\sum_{w \in \Sigma^{*},\sigma \in \Sigma} F_w K_{\sigma} \z_{\sigma w}^{\y}(t)$,
 where $F_i=(A_i-K_iC)$, $i=1,\ldots,\pdim$
 is absolutely convergent in the mean-square sense, hence
 $\bar{\x}(t)-\tilde{\x}(t)$ converges to zero in the mean square sense. 
 It remains to show that 
 $\tilde{\x}(t)=\x(t)$. To this end, notice that 
 $\begin{bmatrix} \tilde{\x}^T & \y^T \end{bmatrix}^T$ is ZMWSII by Lemma \ref{proof:lemma1}, 
 the elements of $\tilde{\x}$ belong to $\mathcal{H}^{\y}_t$ and it satisfies
 \[
   \begin{split}
    & \tilde{\x}(t+1)=\sum_{i=1}^{\pdim}
     ((A_i-K_iC)\tilde{\x}(t)+K_i\y(t))\p_i(t)
    \end{split} 
 \]    
 At the same time, $\begin{bmatrix} \x^T & \y^T \end{bmatrix}^T$ is ZMWSII. Note that the components of
 $\x$ belong to $\mathcal{H}^{\y}$: as it was pointed out in the proof of Lemma \ref{proof:lemma1}, 
 $\mathcal{H}^{\y}_t$ equals the Hilbert-space generated by $\{\z_w^{\e}(t)\}_{w \in \Sigma^{+}}$ and
 the components of $\x(t)$ belong to the latter Hilbert-space. 
 Hence, $\x(t)$ satisfies \eqref{gen:filt:bil:def:pred}, hence
 by Lemma \ref{proof:lemma1} $\x(t)=\bar{\x}(t)$.
\end{proof}

\begin{proof}[Proof of Theorem \ref{inv:gbs:lemma}]
    Note that we can write $\vb(t)=\y(t)-C\x(t)$ and hence the first equation of \eqref{eq:aslpv} holds, i.e., 
    $\x(t+1) = \sum_{i=1}^{\pdim}  (A_i-K_iC)\x(t)+K_i\y(t))\bmu_i(t)$. 
    Since the matrix $\sum_{i=1}^{\pdim} \bmu_i (A_i-K_iC) \otimes (A_i-K_iC)$ is stable, then by repeating the steps of the proof of \cite[Lemma1]{PetreczkyBilinear} it can be shown that
   $\x(t)=\sum_{w \in \Sigma^{*},i \in \Sigma} \sqrt{p_{iw}} A_wK_i\z^{\y}_{iw}(t)$, and hence the elements of $\x(t)$ belong to the Hilber-space generated by $\{\z^{\y}_w(t)\}_{w \in \Sigma^{+}}$.
   Note that $E[\vb(t) \mid \{\z^{\y}_w(t)\}_{w \in \Sigma^{+}}]=0$, see the proof of \cite[eq. (37), proof of Theorem 4]{PetreczkyBilinear}, hence,
   $E[\y(t) \mid \{\z^{\y}_w(t)\}_{w \in \Sigma^{+}}]=C\x(t)$ and therefore $\e(t)=\vb(t)$.
  \end{proof}

\begin{proof}[Proof of Theorem \ref{min:forw:gbs:lemma}]
   Note that $\mathcal{S}$ is minimal if and only if the observability and reachability matrices satisfy the following rank conditions $\rank~ \mathscr{O}_{n-1}(\mathcal{S})=n$ and $\rank~ \mathscr{R}_{n-1}(\mathcal{S})=n$ . 
   Note that the rows of the extended observability matrix $\mathcal{O}_{n-1}$ of the associated dLPV-SSA $\mathcal{D}_{\mathcal{S}}$ are either zero or they coincide with the rows of the observability matrix $\mathscr{O}_{n-1}(\mathcal{S})$, i.e., $\rank~ \mathscr{O}_{n-1}(\mathcal{S})=\rank~ \mathcal{O}_{n-1}$. 
   That is, $\mathcal{S}$ satisfies the observability rank condition if and only if the dLPV-SSA $\mathcal{D}_{\mathcal{S}}$ is observable.
   We will show that $\IM \mathscr{R}_{n-1}(\mathcal{S})=\IM \mathcal{R}_{n-1}$, where $\mathcal{R}_{n-1}$ is the extended controllability matrix of the dLPV-SSA $\mathcal{D}_{\mathcal{S}}$. 
   From this, it follows that $\mathcal{S}$ satisfies the reachability rank condition if and only if $\mathcal{L}_{\mathcal{S}}$ is span-reachable. 

   Now we will show that $\IM~ \mathscr{R}_{n-1}(\mathcal{S})=\IM~ \mathcal{R}_{n-1}$.
   To this end, we recall that $\x(t)$ belongs to the linear  space generated by the columns of $A_{w}K_{\sigma}$, $w \in \Sigma^{*}$, $\sigma \in \Sigma$.
   Since $B_\sigma=E[\x(t)(\z^{\y}_{\sigma}(t))^T]$, it then follows that the columns of $B_\sigma$ also belong to the linear space generated by the columns of $A_{w}K_\sigma$, $w \in \Sigma^{*}, \sigma \in \Sigma$.
   Therefore, the columns of $A_vB_\sigma$, $v \in \Sigma^{*}, \sigma \in \Sigma$ also belong to the  linear space generated by the columns of $A_{w}K_\sigma$, $w \in \Sigma^{*}, \sigma \in \Sigma$.
   In turn, it is easy to see that latter subspace equals $\IM~ \mathcal{R}_{n-1}$. 
   That is, $\IM~ A_vB_\sigma$ is a subspace of $\IM~ \mathcal{R}_{n-1}$, and therefore $\IM~ \mathscr{R}_{n-1}(\mathcal{S}) \subseteq \IM~ \mathcal{R}_{n-1}$. 
     Conversely, from \cite[eq. (37), proof of Theorem 4]{PetreczkyBilinear} it follows that $E[\x(t)\z^{\y}_{\sigma v}(t)]=\sqrt{p_w} A_vB_\sigma$, i.e., for every $w \in \Sigma^{+}$, the columns
   $E[\x(t)(\z^{\y}_{w}(t))^T]$  belong to the space generated by $A_vB_\sigma$, $\sigma \in \Sigma$, $v \in \Sigma^*$. Notice that by \cite[Theorem 2 and Remark 1]{PetreczkyRealChapter}
applied to the LSS\ $\Sigma_{\mathcal{S}}$, the latter space equals $\IM \mathscr{R}_{n-1}(\mathcal{S})$. 
   Since the elements of $\z^{\e}_{\sigma}(t)$  are limits of finite linear combinations of the rows of  $\{\z^{\y}_{w}(t)\}_{w \in \Sigma^{+}}$, it then follows that the columns of 
   $E[\x(t)(\z^{\e}(t))^T]$ are the limits of finite linear combinations of columns of $E[\x(t)(\z^{\y}_{w}(t))^T]$, $w \in \Sigma^{+}$, and hence the columns of  $E[\x(t)(\z^{\y}_{w}(t))^T]$, $w \in \Sigma^{+}$
   also belong to $\IM \mathscr{R}_{n-1}(\mathcal{S})$. 
   From \cite[Proof of Theorem 4]{PetreczkyBilinear} it follows that $K_\sigma Q_\sigma=E[\x(t)(\z^{\e}_\sigma(t))^T]$, where $Q_\sigma=E[\e(t)\e^T(t)\bmu_\sigma^2(t)]$, and hence the columns of $K_\sigma Q_\sigma$ belong to
   $\IM \mathscr{R}_{n-1}(\mathcal{S})$. Since $Q_\sigma$ is non-singular, it then follows that the columns of $K_\sigma$ belong to  $\IM \mathscr{R}_{n-1}(\mathcal{S})$. Since $\IM \mathscr{R}_{n-1}(\mathcal{S})$ is $A_\sigma$-invariant for all
   $\sigma \in \Sigma$ and $A_0=0$, it then follows that $\IM A_vK_\sigma \subseteq \IM \mathscr{R}_{n-1}$ for all $v \in \AQ^{*}$, $\sigma \in \AQ$, and thus $\IM \mathcal{R}_{n-1} \subseteq \IM \mathscr{R}_{n-1}(\mathcal{S})$.
  \end{proof}
  
  \subsection{Proof of Lemma 2}
  \begin{Definition}[Necessary Maps]
  Define, for the sequel, the following maps:
  $$\mathscr{Y}: (\ub,\p) \longmapsto \yb = \mathscr{Y}(\ub,\p)$$
  $$\mathscr{F}: \mathcal{D} \longmapsto \mathscr{F}(\mathcal{D})$$
  $$\mathscr{D}: \mathcal{D} \longmapsto \mathscr{D}(\mathcal{D})$$
  where $\mathscr{Y}$ is the input-output map, $\mathcal{D} = (\{A_i,K_i\}_{i=1}^{\pdim},C,I)$ is a dLPV-SSA and $\mathscr{F}(\mathcal{D})$ and $\mathscr{D}(\mathcal{D})$ are two transformed dLPV-SSAs with $\mathscr{F}(\mathcal{D})=(\{A_i-K_iC,K_i\}_{i=1}^{\pdim},C,I)$  and $\mathscr{D}(\mathcal{D})=(\{A_i,K_i\}_{i=1}^{\pdim},-C,I)$
  \end{Definition}
  Note that, if $\mathcal{X}$ is a dLPV-SSA representation, then it is clear that $\mathscr{D}(\mathscr{F}(\mathscr{D}(\mathscr{F}(\mathcal{X}))))=\mathcal{X}$.
  
  \begin{Lemma}
  \label{new:lem}
  Consider the following dLPV-SSA:
  $\mathcal{D} = (\{A_i,K_i\}_{i=1}^{\pdim},C,I)$ where it is a realization of the sub-Markov parameters $M_\mathcal{D} = CA_wK_\sigma$ where $w \in \Sigma^{*}$, $\sigma \in \Sigma$. Subsequently, $\mathscr{F}(\mathcal{D})$ is a realization of $M_{\mathscr{F}(\mathcal{D})} = C\Tilde{A}_wK_\sigma$ with $\Tilde{A}_i = A_iK_iC$. If a dLPV-SSA $\mathcal{D}' = (\{A'_i,K'_i\}_{i=1}^{\pdim},C',I)$ is a realization of $M_\mathcal{D}$, then the $\mathscr{F}(\mathcal{D'})$ dLPV-SSA is a realization of $M_{\mathscr{F}(\mathcal{D})}$.
  \end{Lemma}
  
  \begin{proof}[Proof of Lemma \ref{new:lem}]
     Consider that $(\x,\ub,\p,\y)$ is a solution of the system $\mathcal{D}$ it then follows that 
     \begin{align*}&\x(t+1) = \sum_{i=1}^{\pdim}(A_i\x(t)+K_i\ub(t))\p_i(t) \\
     &\y(t)=C\x(t)\end{align*}
     It can then be modified to 
     $$\x(t+1)=\sum_{i=1}^{\pdim}(A_i\x(t)+K_i(C\x(t) - C\x(t) + \ub(t)))\p_i(t)$$
     $$\x(t+1)=\sum_{i=1}^{\pdim}((A_i-K_iC)\x(t)+K_i(\yb(t) + \ub(t)))\p_i(t)$$
     Subsequently, it is safe to say that $(\x,\vb,\p,\y)$ is a solution of the system $\mathscr{F}(\mathcal{D})$ with $\vb = \yb + \ub$.
     It can also be shown, with the same demonstration, that if $(\x,\vb,\p,\yb)$ is a solution of $\mathscr{F}(\mathcal{D})$, then $(\x,\ub,\p,\yb)$ is a solution of $\mathcal{D}$, with $\ub=\vb-\yb$.
     That is, it can be said that $\mathscr{Y}_{\mathcal{D},0}(\ub,\p)=\yb=\mathscr{Y}_{\mathscr{F}(\mathcal{D}),0}(\vb,\p)$.
     Consider now that $(\Tilde{\xb},\ub,\p,\Bar{\yb})$ is a solution of $\mathcal{D}'$, it is safe to say that $(\Tilde{\xb},\Bar{\vb},\p,\Bar{\yb})$ is a solution of $\mathscr{F}(\mathcal{D}')$, with $\Bar{\vb}=u+\Bar{\yb})$.
     Lemma \ref{new:lem} assumes that $\mathcal{D}$ and $\mathcal{D}'$ are both realization of the sub-Markov parameters $M_\mathcal{D}$. This leads to the following conclusion:
     $$\mathscr{Y}_{\mathscr{F}(\mathcal{D}),0}(\vb,\p)=\mathscr{Y}_{\mathcal{D},0}(\ub,\p)=\yb=\Bar{\yb}=\mathscr{Y}_{\mathcal{D}',0}(\ub,\p)=\mathscr{Y}_{\mathscr{F}(\mathcal{D}'),0}(\bar{\vb},\p)$$
     In other words, $\mathscr{F}(\mathcal{D}')$ is a realization of $M_{\mathscr{F}(\mathcal{D})}$.
  \end{proof}
  
  \begin{Corollary}
  An dLPV-SSA representation $\Sigma$ is a minimal realization of $M_\Sigma$ if and only if $\mathscr{F}(\Sigma)$ is a minimal realization of $M_{\mathscr{F}(\Sigma)}$.
  \end{Corollary}
  
  \begin{proof}
     Suppose that $\Sigma$ is minimal and $\mathscr{F}(\Sigma)$ is not. Which means that, there exists a minimal dLPV-SSA representation $\Sigma^m$ such that $dim\Sigma^m < dim\mathscr{F}(\Sigma) = dim\Sigma$.
     Define a dLPV-SSA representation $\Hat{\Sigma}^m := \mathscr{D}(\mathscr{F}(\mathscr{D}(\Sigma^m)))$. 
     It is known that $\Sigma^m$ and $\mathscr{F}(\Sigma)$ are both a realization of $M_{\Sigma^m} = M_{\mathscr{F}(\Sigma)}$, then, by applying $\mathscr{D}(\mathscr{F}(\mathscr{D}(\cdot)))$ on both sides, we will get the following:
     \begin{equation*}
         M_{\mathscr{D}(\mathscr{F}(\mathscr{D}(\Sigma^m)))} = M_{\mathscr{D}(\mathscr{F}(\mathscr{D}(\mathscr{F}(\Sigma))))} \Longrightarrow M_{\Sigma} = M_{\Hat{\Sigma}^m}
     \end{equation*}
     Which means that $\Hat{\Sigma}^m$ is a realization of $M_\Sigma$ and $dim\Sigma = dim\mathscr{F}(\Sigma) > dim\Sigma^m = dim\Hat{\Sigma}^m$,
     which implies that $\Sigma$ is not minimal, and results to a contradiction.
  \end{proof}
  
  \begin{proof}[Proof of Lemma \ref{alg:min1:lem}]
     In order to prove Lemma \ref{alg:min1:lem}, two steps are needed: \textbf{(1)} proving that $\mathcal{S}_m$ is stably invertable, and \textbf{(2)} proving that $\mathcal{S}_m$ is a realization of $(\yb,\p)$ in innovation form.
     \newline
     \textbf{(1)} From Lemma \ref{new:lem} it follows that $\mathscr{F}(\mathcal{D})$ and $\mathscr{F}(\mathcal{D}_m)$ are both a realization of $M_{\mathcal{F}(\mathcal{D})}$, this is due to the fact that $\mathcal{D}$ and $\mathcal{D}_m$ are both a realization of $M_\mathcal{D}$.
     Recall from \cite[Lemma 6]{PetreczkyNAHS}\footnote{The termenology of \cite{PetreczkyNAHS} refers to the switch systems representations not the LPV-SSA representations} that if the matrix $\sum_{i=1}^{\pdim}(\sqrt{p_i}\mathscr{A}_i^T)^T\otimes (\sqrt{p_i}\mathscr{A}_i^T)^T = \sum_{i=1}^{\pdim}p_i\mathscr{A}_i\otimes \mathscr{A}_i$ is stable then the matrix $\sum_{i=1}^{\pdim}p_i\mathscr{A}_i^m\otimes \mathscr{A}_i^m$ is also stable, where $\mathscr{A}_i$ are the state matrices of an dLPV-SSA $\Sigma_1$ and $\mathscr{A}_i^m$ are the state matrices of the minimized dLPV-SSA $\Sigma_1^m$.
     This said, recall that Algorithm \ref{alg:min1} assumes that its input $\mathcal{S}=(\{A_{\sigma},K_{\sigma}\}_{\sigma=1}^{\pdim},C,I_{\ny})$ is stably invertable. 
     It follows that the matrix $\sum_{i=1}^{\pdim} p_i (A_i-K_iC) \otimes (A_i-K_iC)$ is stable, where $(A_i-K_iC)$ are the state matrices of $\mathscr{F}(\mathcal{D})$. 
     Note that $\mathscr{F}(\mathcal{D}_m)$ is minimal due to the minimality of $\mathcal{D}_m$. 
     It follows, from \cite[Lemma 6]{PetreczkyNAHS}, that the matrix $\sum_{i=1}^{\pdim} p_i (A_i^m-K_i^mC^m) \otimes (A_i^m-K_i^mC^m)$ is stable, where $(A_i^m-K_i^mC^m)$ are the state matrices of $\mathscr{F}(\mathcal{D}_m)$. 
     Therefore, the algorithm's output $\mathcal{S}_m$ is indeed stably invertable. 
     \newline
     \textbf{(2)} Since $\mathcal{S}$ is a realization of $(\yb,\p)$ in innovation form, then there exists a process $\xb$ such that \eqref{eq:aslpv} holds with $\vb=\yb$.
     This said, let us apply the linear transformation $T$ as described in \cite[Corollary 1]{PetreczkyLPVSS} to $\mathcal{S}$. Let $\Hat{\xb}(t) = T\xb(t)$, $\Hat{A}_i = TA_i$, $\Hat{K}_i = TK_i$, for $i=1,\ldots,\pdim$, and $\Hat{C} = CT^{-1}$. Then the asLPV-SSA $\Hat{\mathcal{S}} = (\{\Hat{A}_i,\Hat{K}_i\}_{i=1}^{\pdim},\Hat{C},I,\eb)$ is also a realization of $(\yb,\p)$ in innovation form.
     Now, in order to get the Kalman decomposition, recall form \cite[Corollary 1]{PetreczkyLPVSS} that:
     \begin{equation*}
         \Hat{A}_i=\begin{bmatrix}
              A_i^m & 0 & A_i' \\
              A_i'' & A_i''' & A_i'''' \\
              0 & 0 & A_i^{uc} \\
         \end{bmatrix}, \quad
         \Hat{K}_i = \begin{bmatrix}
            K_i^m \\
            K_i' \\
            0 \\
         \end{bmatrix}, \quad
         \Hat{C} = \begin{bmatrix} 
            C^m & 0 & C' \\
         \end{bmatrix}
     \end{equation*}
     for $i=1,\ldots,\pdim$, and for suitable block matrices. 
     In particular, for all $w \in \Sigma^*$ we found:
     \begin{equation}
     \label{Aw_matrix}
         \Hat{A}_w = \begin{bmatrix}
              A_w^m & 0 & \ast \\
              \ast & A_w''' & \ast \\
              0 & 0 & A_w^{uc} \\
         \end{bmatrix}
     \end{equation}
     where $\ast$ refers to a block matrix. 
     Indeed, it can be shown that for $w=\epsilon$, $\Hat{A}_w = I$.
     Now if \eqref{Aw_matrix} holds for $w=v$, then for $w=v\sigma$:
     \begin{equation*}
         \Hat{A}_w = \begin{bmatrix}
              A_\sigma^m & 0 & \ast \\
              \ast & A_\sigma''' & \ast \\
              0 & 0 & A_\sigma^{uc} \\
         \end{bmatrix}\begin{bmatrix}
              A_v^m & 0 & \ast \\
              \ast & A_v''' & \ast \\
              0 & 0 & A_v^{uc} \\
         \end{bmatrix} = \begin{bmatrix}
              A_{\sigma v}^m & 0 & \ast \\
              \ast & A_{\sigma v}''' & \ast \\
              0 & 0 & A_{\sigma v}^{uc} \\
         \end{bmatrix}
     \end{equation*}
     Hence, by induction of the length of $w$, \eqref{Aw_matrix} holds. It the follows that:
     \begin{equation*}
         \Hat{A}_w\Hat{K}_\sigma = \begin{bmatrix}
         A_w^mK_\sigma^m \\
         \ast \\
         0 \\
         \end{bmatrix}
     \end{equation*}
     Consider the following decomposition:
     \begin{equation*}
         \Hat{\xb}(t)=T\xb(t)=\begin{bmatrix}
         \xb^m(t) \\
         \xb^c(t) \\
         \xb^{uc}(t) \\
         \end{bmatrix}
     \end{equation*}
     From Lemma \ref{proof:lemma-2}, $\Hat{\xb}(t)$ can be expressed as:
     \begin{equation*}
         \Hat{\x}(t)=\sum_{w \in \Sigma^{*},\sigma \in \Sigma} \sqrt{p_\sigma w} \Hat{A}_w \Hat{K}_{\sigma} \z_{\sigma w}^{\eb}(t)
     \end{equation*}
     It follows that:
     \begin{equation}
     \label{x:decomp}
     \Hat{\x}(t)=\sum_{w \in \Sigma^{*},\sigma \in \Sigma} \sqrt{p_\sigma w} \begin{bmatrix} A_w^m K^m_{\sigma} \z_{\sigma w}^{\eb}(t) \\
     \ast \\
     0 \\
     \end{bmatrix} = \begin{bmatrix} \sum_{w \in \Sigma^{*},\sigma \in \Sigma} \sqrt{p_\sigma w} A_w^m K^m_{\sigma} \z_{\sigma w}^{\eb}(t) \\
     \ast \\
     0 \\
     \end{bmatrix}
     \end{equation}
     From \eqref{x:decomp} it follows that $\xb^m(t)= \sum_{w \in \Sigma^{*},\sigma \in \Sigma} \sqrt{p_\sigma w} A_w^m K^m_{\sigma} \z_{\sigma w}^{\eb}(t)$. 
     Then from \cite[Lemma 3]{PetreczkyBilinear}, it follows that $\xb^m(t)$ is the state process of $\mathcal{S}_m$. 
     Finally, notice that, as $\Hat{\mathcal{S}}$ is a realization of $(\yb,\p)$, $\yb(t) = C\xb(t) + \eb(t) = \Hat{C}T\xb(t) + \eb(t) = \Hat{C}\Hat{\xb}(t) + \eb(t)$. The output can then be expressed as:
     \begin{equation}
     \label{y(t):min}
         \yb(t) = \begin{bmatrix} C^m & 0 & C' \end{bmatrix}\begin{bmatrix} \xb^m(t) \\ \xb^c(t) \\ \xb^{uc}(t) \\ \end{bmatrix} + \eb(t)
     \end{equation}
     From \eqref{x:decomp} it follows that $\xb^{uc}(t)=0$. Hence, from \eqref{y(t):min} it follows that $y(t) = C^m\xb^m(t) + e(t)$.
     Since $\xb^m$ is the state process of $\mathcal{S}_m$, it then follows that $\mathcal{S}_m$ is a realization of $(\yb,\p)$. 
     In addition, because $\mathcal{S}_m$ is stably invertable, as proven in \textbf{(1)}, then $\mathcal{S}_m$ is a realization of $(\yb,\p)$ in innovation form.
     \end{proof}

\end{document}